\documentclass{article}
\usepackage{amsfonts, amssymb, amsmath, amsthm, enumerate}
\usepackage{color,graphicx}

\newtheorem{thm}{Theorem}[section]
\newtheorem{prop}[thm]{Proposition}
\newtheorem{cor}[thm]{Corollary}
\newtheorem{lem}[thm]{Lemma}

\newtheorem{conj}[thm]{Conjecture}

\newcommand{\II}{{\mathbb{I}}}
\newcommand{\NN}{{\mathbb{N}}}
\newcommand{\ZZ}{{\mathbb{Z}}}

\newcommand{\RR}{{\mathbb{R}}}

\newcommand{\cB}{{\mathcal{B}}}
\newcommand{\cC}{{\mathcal{C}}}
\newcommand{\cD}{{\mathcal{D}}}
\newcommand{\cS}{{\mathcal{S}}}
\newcommand{\bb}{{\mathbf{b}}}
\newcommand{\bc}{{\mathbf{c}}}
\newcommand{\bd}{{\mathbf{d}}}
\newcommand{\bg}{{\mathbf{g}}}
\newcommand{\bx}{{\mathbf{x}}}
\newcommand{\by}{{\mathbf{y}}}
\newcommand{\bz}{{\mathbf{z}}}
\newcommand{\cl}{c}
\newcommand{\lh}{l}

\newcommand{\diam}{\operatorname{diam}}
\newcommand{\xl}{\operatorname{xl}}
\newcommand{\st}{\operatorname{st}}
\newcommand{\ga}{\operatorname{ga}}

\newcommand{\Vol}{\operatorname{Vol}}
\def\Pr{\mathop{\rm Pr}\nolimits}

\newcommand{\mi}{I}
\newcommand{\mj}{J}

\begin{document}
\title{Scaling limits of discrete copulas \\
are bridged Brownian sheets}
\author{Juliana Freire, Nicolau C. Saldanha and Carlos Tomei}
\maketitle

\begin{abstract}
For large $n$, take a random $n \times n$ permutation matrix
and its associated discrete copula $X_n$.
For $a, b = 0, 1, \ldots, n$, 
let $\by_n(\frac{a}{n},\frac{b}{n}) =
\frac{1}{n} \left(X_{a,b} - \frac{ab}{n}\right)$;
define $\by_n: [0,1]^2 \to \RR$ by
interpolating quadratically on squares of side $\frac{1}{n}$.
We prove a Donsker type central limit theorem:
$\sqrt{n} \by_n$ approaches a bridged Brownian sheet on the unit square.
\end{abstract}

\bigskip

\begin{narrower}



\end{narrower}

\bigskip

\section{Introduction}
\label{sect:intro}

Mostly following \cite{MST}, a \textit{discrete copula}
is a square matrix $C$ with indices $a, b \in \{0, 1, \ldots, n\}$
and \textit{integer} entries satisfying
\begin{enumerate}[(i)]
\item{$C_{a,0}=C_{0,b}=0$ for all $a, b$;}
\item{$C_{a,n}=C_{n,a}=a$ for all $a$;}
\item{$C$ is $2$-increasing, i.e., 
if $0 \le a \le a' \le n$ and $0 \le b \le b' \le n$, then
\[ C_{a',b'} + C_{a,b} \ge C_{a',b} + C_{a,b'}. \]}
\end{enumerate}
Let $S_n$ be symmetric group
interpreted as the set of all $n \times n$ permutation matrices;
also, denote the set of discrete copulas by  $\cS_n$.
As is well known (\cite{MST}),
there is a natural bijection $S_n \to \cS_n$
with $M \mapsto C$ where $C$ is obtained by (discrete) integration:
\[ C_{a,b} = \sum_{a' \le a; \; b' \le b} M_{a',b'}\,. \]
Notice that if $a = 0$ or $b = 0$ the sum is empty and $C_{a,b} = 0$.

A continuous copula is a $1$-Lipschitz function $\bc: \II^2 \to \II$
which describes the joint distribution
of two real-valued random variables $\bz_0$ and $\bz_1$,
assumed here to be both uniform in $\II = [0,1]$
(for basic information about copulas, see \cite{Nelsen}).
If $\bz_0$ and $\bz_1$ are independent then 
$\bc$ is the product copula $\bc_0(u,v) = uv$.
Being associated with a permutation matrix,
a discrete copula describes the joint distribution
of two discrete random variables $Z_0$ and $Z_1$,
both uniform in $\{1, 2, \ldots, n\}$,
when one is a function of the other.
Let $C_0$, the \textit{product copula}, be the baricenter of $\cS_n$,
so that $(C_0)_{a,b} = \frac{ab}{n}$;
it is analogous to the (continuous) product copula
and indicates that $Z_0$ and $Z_1$ are independent.
Clearly, $C_0$ is not a discrete copula:
it is an example of Birkhoff copula,
a concept to which we shall return in Section \ref{sect:birkhoff}.
We scale each discrete copula $C$ 
to obtain a continuous copula $\bc: \II^2 \to \II$ satisfying
\[ \bc\left(\frac{a}{n}, \frac{b}{n}\right) = \frac{C_{a,b}}{n}, \]
and interpolated on small squares
$[\frac{a-1}{n},\frac{a}{n}] \times [\frac{b-1}{n},\frac{b}{n}]$
by a polynomial of the form
$\bc(u,v) = a_{00} + a_{10} u + a_{01} v + a_{11} uv$.
The statistics of the differences $D = C - C_0$ and $\bd = \bc - \bc_0$,
appropriately scaled, are the main subject of this paper.

\begin{figure}[t]
\begin{center}
\includegraphics[height=39mm,angle=0]{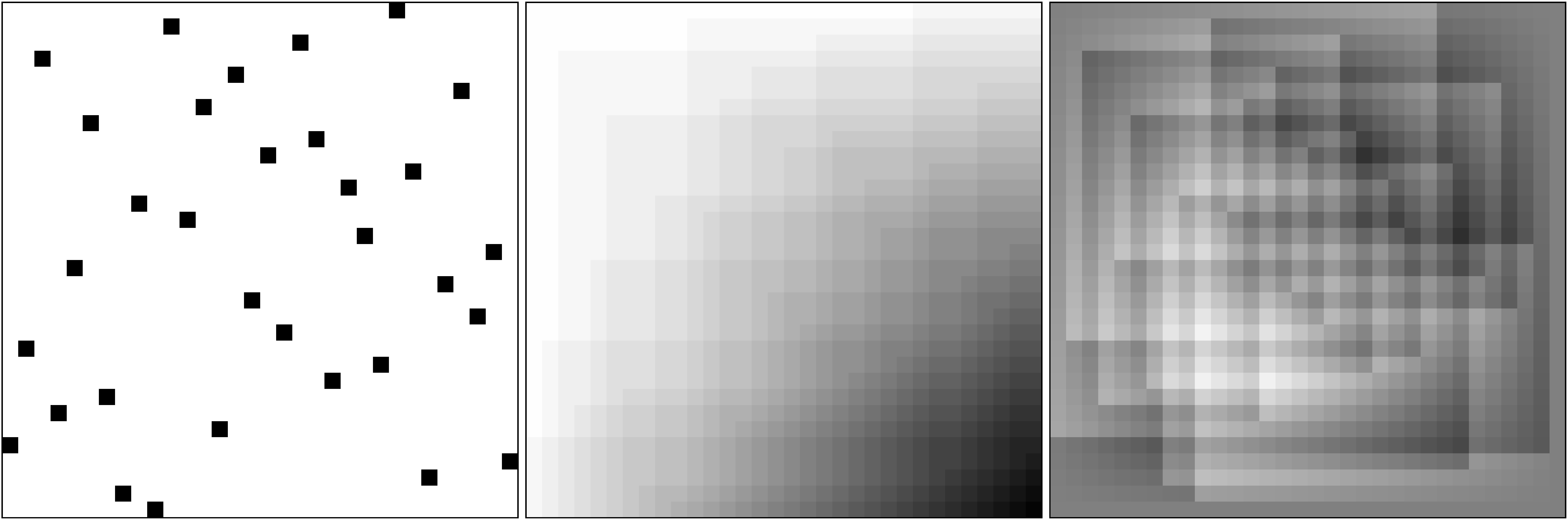}
\includegraphics[height=39mm,angle=0]{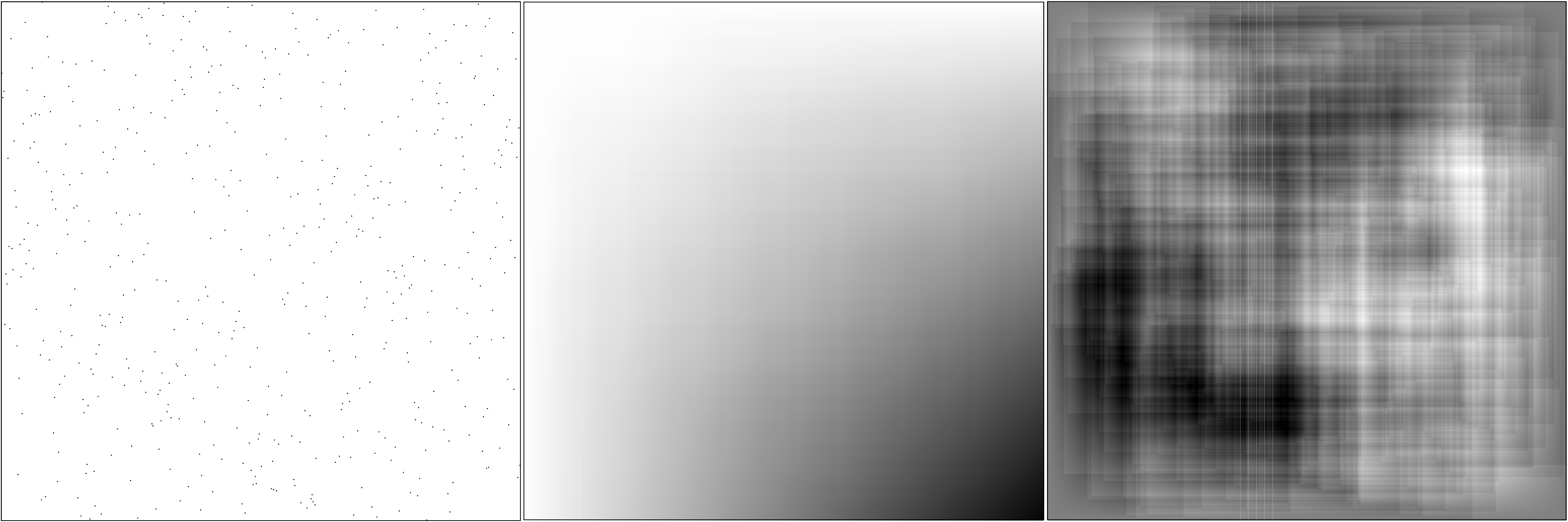}
\end{center}
\caption{Two permutation matrices, their associated discrete copula $C$
and the (scaled) difference $D = C - C_0$; $n = 32$ and $n = 512$.}
\label{fig:shade}
\end{figure}

In Figure \ref{fig:shade}, the nonzero entries of a permutation matrix
are represented by solid black squares. For the associated discrete copula
$C$, each entry is represented by a grey square, with $0$ being white 
and $n$ (the largest value) black.
Finally, for the difference $D$,
entries with $D_{i,j} \ge \sqrt{n}/2$ are painted black
and with $D_{i,j} \le -\sqrt{n}/2$, white.

One of our aims is to obtain a concentration result:
a typical discrete copula is near $C_0$.
Indeed, in Figure \ref{fig:shade} the central squares
are almost indistinguishable.
We shall consider discrete random variables $X_n$ assuming matrix values
in the finite set $\cS_n$
with uniform distribution and entries $(X_n)_{a,b} = X_{n,a,b}$.
We also consider the differences $Y_n = X_n - C_0$.
Just as we extended $C$ or $D$ to $\bc$ or $\bd$,
we define the extensions
$\bx_n \in C^0(\II^2)$ and $\by_n \in C^0_0(\II^2)$
(i.e., $\by_n|_{\partial\II^2} = 0$) of $X_n$ and $Y_n$.

We shall see that,
with high probability, the values assumed by $\by_n$
have order at most $1/\sqrt{n}$.
More: $\sqrt{n}\;\by_n$ approaches a bridged Brownian sheet,
compatibly with the patterns
in the rightmost squares in Figure \ref{fig:shade}.

\medskip

Recall that a \textit{Brownian sheet} is a process
yielding continuous functions $f_\ast: \II^2 \to \RR$
(see \cite{Leandre});
this  generalizes to a two-dimensional domain
the more familiar Brownian motion or Wiener process
(see \cite{DudleyRAP}).
In order to select $f_\ast$, first generate white noise $\eta$ in
the unit square $\II^2$ and then integrate it:
\[ f_\ast(u,v) = \iint_{[0,u] \times [0,v]} \eta\,. \]
In particular, $f_\ast|_{ \left( (\II \times \{0\})
\cup (\{0\} \times \II) \right) } = 0$.
We are interested in the \textit{bridged Brownian sheet},
a modified process which obtains functions $f \in C^0_0(\II^2)$.
In order to select $f$ from $f_\ast$, define
\begin{equation}
\label{eq:ffromfast}
f(u,v) = f_\ast(u,v) - u f_\ast(1,v)
- v f_\ast(u,1) + u v f_\ast(1,1).
\end{equation}
It turns out that the values of $f$ on a fixed set of points follow
a joint normal distribution (see Equation \ref{eq:jointbrown}).
We are particularly interested in the square grid $\mi = \mj = m$,
$u_i = \frac{i}{m}$, $v_j = \frac{j}{m}$.
Large permutation matrices satisfy
similar statistics on coarse $m\times m$ grids:
for $\alpha = {19}/{20}$,
if $m < n^{(1-\alpha)}$ then the function $\sqrt{n}\;\by_n$ 
behaves as $f$ in such grids
(Lemmas \ref{lem:begood} and \ref{lem:goodnormal}).
Still, the grid is fine enough for errors to be controlled
(Lemma \ref{lem:holder}).

\medskip

We are ready to state our main result.
As with Donsker's theorem,
this is a kind of central limit theorem,
in which a Brownian bridge or bridged sheet is obtained
as a limit of discrete processes.
We use Dudley's framework as in \cite{DudleyUCLT}.

\begin{thm}
\label{theo:discrete}
There exist probability spaces $\Omega_n$ with the following properties.
\begin{enumerate}[(a)]
\item{On $\Omega_n$ there exists a random variable
$X_n$ uniformly distributed in $\cS_n$.}
\item{On $\Omega_n$ a bridged Brownian sheet process
$f$ in $C^0_0(\II^2)$ is defined.}
\item{For all $\epsilon > 0$,
\[ \lim_{n \to \infty}
\Pr\left[ \left\| \sqrt{n}\;{\by_n} - f \right\|_{C^0} >
\epsilon \right] = 0, \] 
where $\by_n$ is defined from $X_n$ as above. }
\end{enumerate}
\end{thm}


\medskip

In Section 2 we study the pointwise behavior of $\by$
for the case of discrete copulas.
One of the main results is Lemma \ref{lem:prenormal};
its Corollary \ref{coro:normal} shows that, given $(u, v) \in \II^2$,
the distribution of $\sqrt{n}\;\by_n(u, v)$ approaches a normal distribution,
in accordance with the results for bridged Brownian sheets
mentioned above.
In Section 3 we state and prove Theorem \ref{theo:joint}:
if $m \in \NN$ and points
$(u_1,v_1), \ldots, (u_m,v_m) \in \II^2$ are fixed
and $n$ goes to infinity, then the joint distribution
of $\sqrt{n}\;\by_n(u_i, v_i)$ converges to the corresponding
distribution for bridged Brownian sheets.
Section 4 is centered on Lemma \ref{lem:holder},
a H\"older type estimate for $| D_{a_2,b_2} - D_{a_1,b_1} |$
provided $|a_2 - a_1|$ and $|b_2 - b_1|$ are small;
this goes together with the fact that Brownian sheets
are H\"older continuous for any exponent smaller than $\frac12$
(see Lemma \ref{lem:holderbrownsheet}).
In Section 5 we complete the proof of Theorem \ref{theo:discrete}:
the key difference from Section 3 is that now the size of the grid
goes to infinity as $m \approx n^{\frac{1}{20}}$.
Section 6 contains, as an application of Theorem \ref{theo:discrete},
a concentration result which makes no reference to bridged Brownian sheets.
In Section 7 we briefly discuss a formulation of our result
in terms of lozenge tilings of a regular hexagon.
Finally, in Section 8, we present a conjecture analogous to 
Theorem \ref{theo:discrete} for Birkhoff copulas,
which are real valued matrices with a definition similar to that
of discrete copulas.

The authors are thankful for the generous support of CNPq, CAPES and FAPERJ
and for fruitful conversations with Felipe Pina.

\section{Pointwise convergence}
\label{sect:pointwise}

For a positive integer $n$,
let $X_n$ be a random discrete copula with uniform distribution,
that is, for each copula $C \in \cS_n$
(or, equivalently, for each permutation matrix $M \in S_n$)
we have $\Pr[X_n=C]=1/n!$.
Given $a, b, \cl \in \NN$, write 
\[ P(n,a,b,\cl) = \Pr[X_{n,a,b} = \cl], \]
the probability that the upper left $a \times b$ submatrix of $M \in S_n$
contains precisely $\cl$ entries equal to $1$.
We introduce some heavy handed notation preparing the way to grids,
to be considered in later sections.

\begin{figure}[ht]
\[
\left(
\begin{array}{ccccc|ccc}
0 & 0 & 0 & 1 & 0 & 0 & 0 & 0 \\
1 & 0 & 0 & 0 & 0 & 0 & 0 & 0 \\
0 & 0 & 0 & 0 & 0 & 1 & 0 & 0 \\
0 & 0 & 1 & 0 & 0 & 0 & 0 & 0 \\
\hline \\[-8pt]
0 & 0 & 0 & 0 & 1 & 0 & 0 & 0 \\
0 & 1 & 0 & 0 & 0 & 0 & 0 & 0 \\
0 & 0 & 0 & 0 & 0 & 0 & 0 & 1 \\
0 & 0 & 0 & 0 & 0 & 0 & 1 & 0 
\end{array}\right)\quad
\left(
\begin{array}{ccccc|ccc}
0 & 0 & 0 & 1 & 1 & 1 & 1 & 1 \\
1 & 1 & 1 & 2 & 2 & 2 & 2 & 2 \\
1 & 1 & 1 & 2 & 2 & 3 & 3 & 3 \\
1 & 1 & 2 & 3 & 3 & 4 & 4 & 4 \\
\hline \\[-8pt]
1 & 1 & 2 & 3 & 4 & 5 & 5 & 5 \\
1 & 2 & 3 & 4 & 5 & 6 & 6 & 6 \\
1 & 2 & 3 & 4 & 5 & 6 & 6 & 7 \\
1 & 2 & 3 & 4 & 5 & 6 & 7 & 8 
\end{array}\right)
\]
\caption{A permutation matrix and its discrete copula:
$C_{4,5} = 3$.}
\label{fig:8x8copula}
\end{figure}

A quadruple $\bb = (n,a,b,\cl) \in \NN^4$ is a \textit{blocking}
if it satisfies the following condition:
\begin{equation}
n > 0, \quad
0 \le a, b \le n, \quad
\max(0,n-a-b) \le \cl \le \min(a,b).
\label{eq:nabl}
\end{equation}
A blocking $\bb = (n,a,b,\cl)$ describes a partitioning
of the $n\times n$ square into four rectangles
of sizes $\dot a_i \times \dot b_j$, $1 \le i, j \le 2$, where
\[ \dot a_1 = a_1 = a, \quad \dot a_2 = n - a, \quad
\dot b_1 = b_1 = b, \quad \dot b_2 = n - b, \]
together with the number of nonzero entries $\dot \cl_{i,j}$, where
\[ \dot \cl_{1,1} = \cl_{1,1} = \cl, \quad \dot \cl_{1,2} = a - \cl, \quad
\dot \cl_{2,1} = b - \cl, \quad \dot \cl_{2,2} = n - a - b + \cl, \]
to fit the $(i,j)$-rectangle for any matrix $M \in S_n$
with $C_{i,j} = \cl$.  
In Figure \ref{fig:8x8copula},
\[ \dot \cl_{1,1} = 3, \quad \dot \cl_{1,2} = 1, \quad
\dot \cl_{2,1} = 2, \quad \dot \cl_{2,2} = 2 \]
are indeed the number of nonzero entries in each box.
Notice that Condition (\ref{eq:nabl}) is equivalent to
$\dot \cl_{i,j} \ge 0$ (for all $i, j$).

An expression for this probability 
is a simple combinatorial problem.

\begin{lem}
\label{lem:exact}
Let $\bb = (n, a, b, \cl)$ be a blocking;
then
\begin{gather*}
E[X_{n,a,b}]=(C_0)_{a,b} = \frac{ab}{n}; \\
P(\bb) = P(n,a,b,\cl) =
\frac{a!b!(n-b)!(n-a)!}{n!\cl!(a-\cl)!(b-\cl)!(n-b-a+\cl)!} =
\frac{\dot a_1!\dot a_2!\dot b_1!\dot b_2!}%
{n!\dot \cl_{1,1}!\dot \cl_{1,2}!\dot \cl_{2,1}!\dot \cl_{2,2}!}. 
\end{gather*}
\end{lem}

\begin{proof}
For the expected value,
each one of the first $a$ rows contributes with a $1$ with 
probability $b/n$, obtaining the desired result.

There are $a$ columns in the matrix $M \in S_n$
with a $1$ in the first $a$ rows.
The total number of ways 
to choose those columns is $\binom{n}{a}$. In order to have $X_{n,a,b} = \cl$,
$\cl$ columns must be chosen from the first $b$ columns, and $a-\cl$ columns
from the remaining $n-b$ columns. Thus
\[ P(\bb) 
=\frac{\binom{b}{\cl} \binom{n-b}{a-\cl} }{ \binom{n}{a} }, \]
and the result follows by expanding the binomial coefficients.
\end{proof}

Let $\xl t = t \log t$;
we use Stirling's approximation formula
with tail $\st(k)$:
\begin{equation}
\label{eq:stirling}
\log(k!) = \xl k - k + \frac12 \log k + \frac12 \log(2\pi)
+ \st k, \quad
0 < \st k < \frac{1}{12k}.
\end{equation}


We define the rescalings on $\II$
\[ \hat a_i = \frac{\dot a_i}{n}, \quad
\hat b_j = \frac{\dot b_j}{n}, \quad
\hat \cl_{i,j} = \frac{\dot \cl_{i,j}}{n}, \]
the auxiliary real numbers
\[ \cC = \hat a_1 \hat a_2  \hat b_1 \hat b_2 \in [0,1/16], \quad
\cD = \hat \cl_{1,1} \hat \cl_{1,2} \hat \cl_{2,1} \hat \cl_{2,2}, \]
and the factorial-free approximation
\begin{align*}
\tilde P(\bb) &= 
\frac{\sqrt{\cC}}{\sqrt{2\pi \cD n}}\;
\frac{
{\hat a_1}^{\hat a_1 n} {\hat a_2}^{\hat a_2 n}
{\hat b_1}^{\hat b_1 n} {\hat b_2}^{\hat b_2 n}
}
{ {\hat \cl_{1,1}}^{\hat \cl_{1,1}n}
{\hat \cl_{1,2}}^{\hat \cl_{1,2}n}
{\hat \cl_{2,1}}^{\hat \cl_{2,1}n}
{\hat \cl_{2,2}}^{\hat \cl_{2,2}n}} \\
&= \frac{\sqrt{\cC}}{\sqrt{2\pi \cD n}}\;
\exp\bigg( n\, \bigg(
\sum_{1 \le i \le 2} \xl \hat a_i
+ \sum_{1 \le j \le 2} \xl \hat b_j
- \sum_{1 \le i,j \le 2} \xl \hat \cl_{i,j} 
\bigg) \bigg).
\end{align*}

The \textit{sparsity} of a blocking $\bb = (n,a,b,\cl)$ is
$\lambda(\bb) = \min_{i,j} \dot \cl_{i,j}$.

\begin{lem}
\label{lem:stirling}
\[ \lim_{\lambda(\bb) \to \infty}
\frac{P(\bb)}{\tilde P(\bb)} = 1. \]
\end{lem}

That is, given $\epsilon > 0$ there exists $\lambda_\epsilon$
such that, if $\lambda(\bb) > \lambda_\epsilon$ then
\[ \left| \frac{P(\bb)}{\tilde P(\bb)} - 1 \right| < \epsilon. \]

\begin{proof}
Use Stirling's formula on Lemma \ref{lem:exact}.
Notice that $\lambda(\bb)$ is the smallest integer whose factorial
is taken in the formula of Lemma \ref{lem:exact}.
\end{proof}

We will use another formula for $\tilde P$,
more amenable to Taylor approximations.
From Lemma \ref{lem:exact}, the expected value of
$(X_{n,a,b})/n$, the random variable corresponding to $\hat \cl_{1,1}$,
equals $\hat a \hat b$.
Inspired by this, we introduce new variables
\( \hat \lh = \sqrt{n}\;(\hat \cl - \hat a \hat b) =
(\cl - \frac{ab}{n})/\sqrt{n} \)
and, similarly, $\hat \lh_{i,j}$ such that
\( \hat \cl_{i,j} = \hat a_i \hat b_j + {\hat \lh_{i,j}}/{\sqrt{n}} \).
Notice that
\begin{equation}
\label{eq:sumasumb}
\sum_i \hat a_{i} = \sum_j \hat b_{j} = 1, \quad
\sum_j \hat \cl_{i,j} = \hat a_i, \quad \sum_i \hat \cl_{i,j} = \hat b_j
\end{equation}
and therefore
\begin{equation}
\label{eq:sumh}
\hat \lh_{i,j} = (-1)^{i+j} \hat \lh, \quad
\sum_j \hat \lh_{i,j} = \sum_i \hat \lh_{i,j} = 0.
\end{equation}
We also introduce the corresponding random variables:
\[ \hat h =
\frac{1}{\sqrt{n}} \left(X_{a,b} - \frac{ab}{n}\right); \]
the random variables $\hat h_{i,j}$ are defined similarly;
thus, 
\[ P(\bb) = \Pr[X_{a,b} = \cl] = \Pr[\hat h = \hat\lh]. \]

\begin{lem}
\label{lem:logQ}
\[ \log(\tilde P(\bb)) = - \frac{1}{2} \log\left(
\frac{2\pi \cD n}{\cC} \right) 
- n \sum_{i,j} \hat a_i \hat b_j
\xl\left( 1 + \frac{\hat \lh_{i,j}}{\sqrt{n}\,\hat a_i \hat b_j} \right). \]
\end{lem}

\begin{proof}
Write
\[ \log(\tilde P(\bb)) = - \frac{1}{2} \log\left(
\frac{2\pi \cD n}{\cC} \right) 
- n A \]
where  
\[ A = 
\sum_{i,j} \xl \hat \cl_{i,j} - \sum_i \xl \hat a_i - \sum_j \xl \hat b_j.
\]
We have
\begin{align*}
\xl \hat \cl_{i,j} &=
\left(\hat a_i \hat b_j + \frac{\hat \lh_{i,j}}{\sqrt{n}} \right)
\left( \log \hat a_i  + \log \hat b_j  +
\log\left(1 + \frac{\hat \lh_{i,j}}{\sqrt{n}\,\hat a_i \hat b_j} \right) \right) \\
&= \hat b_j \xl \hat a_i + \hat a_i \xl \hat b_j
+ \frac{\hat \lh_{i,j}}{\sqrt{n}} \log \hat a_i
+ \frac{\hat \lh_{i,j}}{\sqrt{n}} \log \hat b_j 
\\
&\phantom{=} +
\hat a_i \hat b_j
\xl \left( 1 + \frac{\hat \lh_{i,j}}{\sqrt{n}\,\hat a_i \hat b_j} \right).
\end{align*}
From equations \ref{eq:sumasumb} and \ref{eq:sumh} we have
\[ A = \sum_{i,j} \hat a_i \hat b_j
\xl \left( 1 + \frac{\hat \lh_{i,j}}{\sqrt{n}\,\hat a_i \hat b_j} \right). \]
\end{proof}

We expect that appropriate scalings asymptotically yield
normal distributions.
Consider the residuals $Y_n = X_n - C_0$
and their extension $\by \in C^0_0(\II^2)$.
Notice that
the values at two points in the the same small square satisfy
\[ \left| \by(u_0,v_0) - \by(u_1,v_1) \right| < \frac{4}{n}. \]
We shall prove in Corollary \ref{coro:normal}
that, for fixed $(u,v) \in (0,1)^2$ and $n \to \infty$,
$\sqrt{n}\;\by(u,v)$ has normal distribution centered at $0$
with variance $\cC = u(1-u)v(1-v)$.
Consistently with Theorem \ref{theo:discrete},
this is precisely the distribution of $f(u,v)$
if $f$ is a bridged Brownian sheet.

As a warm-up, we compute an approximation
for $P(\bb) = \Pr[X_{n,a,b}=\cl]$ near the mean.

\begin{lem}
\label{lem:praverage}
Let $X_n$ be a random discrete copula with uniform distribution in $\cS_n$.
Given constants $0<u,v<1$,
let $\bb_n = (n,a,b,\cl)$ for
$a=\lfloor u n\rfloor$, $b=\lfloor v n\rfloor$ and
$\cl = \lfloor u v n \rfloor$,
with the other variables defined as above.
Then 
\[ \lim_{n \to \infty}  \sqrt{2\pi \cC n}\;P(\bb_n) = 1. \]
\end{lem}

\begin{proof}
Clearly, the sparsity $\lambda(\bb_n)$ goes to $\infty$ when $n \to \infty$;
in particular, $\bb_n$ satisfies Condition (\ref{eq:nabl}).
In the notation of Lemma \ref{lem:logQ}, we claim that
\[ \lim_{n \to \infty} \frac{\cD}{\cC^2} = 1, \quad
\lim_{n \to \infty} nA = 0, \]
which, given Lemma \ref{lem:stirling}, obtains the desired result
(we occasionally omit the dependency in $n$).
We have $\hat a \to u$ and $\hat b \to v$ so that
$\cC \to u(1-u)v(1-v)$ and $\cD \to (u(1-u)v(1-v))^2$,
from which the first limit follows.
Standard manipulation of integer parts yields
\[ - \frac{2}{\sqrt{n}} < \hat \lh_{i,j} < \frac{2}{\sqrt{n}}. \]
From the Taylor approximation
\begin{equation}
\xl(1+t) = (1+t)\log(1+t) =
t + \frac{t^2}{2} + O(t^3),
\label{eq:taylorxl}
\end{equation}
we may write
\[
nA = n \sum_{i,j} \hat a_i \hat b_j
\xl\left(1+\frac{\hat \lh_{i,j}}{\sqrt{n}\,\hat a_i \hat b_j} \right) 
= \sqrt{n} \sum_{i,j} \hat \lh_{i,j} 
+ \frac{1}{2} \sum_{i,j} \frac{\hat \lh_{i,j}^2}{\hat a_i \hat b_j}
+ O\big(\frac{1}{n^2}\big).
\]
Equations \ref{eq:sumasumb} and \ref{eq:sumh}
imply that the first summation in the expansion equals zero.
The second summation tends to zero.
\end{proof}

We denote the Gaussian centered at $0$ with variance $C$ by
\[ \ga_{C} t  = 
\frac{1}{\sqrt{C}} \ga\left(\frac{t}{\sqrt{C}}\right) =
\frac{1}{\sqrt{C}} \frac{1}{\sqrt{2\pi}}
\exp\left( - \frac{t^2}{2C} \right), \]
with $\ga t = \ga_1 t$; in particular, $\int_{\RR} \ga_{C} = 1$.

The following lemma gives an estimate for $P(\bb)$ for large $n$,
when $a$ and $b$ are not too close to the boundary values $0$ and $n$
(controlled by a parameter $\alpha$) and when
$\cl$ is relatively near its average value $\frac{ab}{n}$
(controlled by $\eta$).
The lemma is rather technical and requires additional notation.

A \textit{grid} is a triple $\bg = (n,a,b)$;
we identify a blocking with a pair $\bb = (\bg,\cl)$.
Given a grid $\bg$, we relate a blocking $\bb = (\bg,\cl)$
with a \textit{normalized blocking} $\hat\bb = (\bg,\hat \lh)$.
Notice that $\hat \lh$ and $\hat h$ must belong to the lattice
\[ L_{\bg} = \frac{1}{\sqrt{n}} \left( \ZZ - \frac{ab}{n} \right)
\subset \RR, \]
with spacing $1/\sqrt{n}$.
We now introduce the first instance of parameters
which shall be restricted to smaller regions along the text.
Let $\alpha \in (\frac{7}{8},1)$; a grid $\bg$ is
\textit{$\alpha$-regular} if
\[ n^\alpha < a, b < n - n^\alpha. \]
Consider the open triangle
\begin{equation}
\label{eq:Delta1}
\Delta_1 =
\{ (\alpha,\eta) \in \RR^2 \,\mid\, 0 < 6\eta < 8\alpha - 7 < 1 \};
\end{equation}
for $(\alpha,\eta) \in \Delta_1$,
a normalized blocking $\hat\bb = (\bg,\hat \lh)$
is \textit{$(\alpha,\eta)$-standard},
or $\hat\bb \in \Xi_{\alpha,\eta}$, if $\bg$ is $\alpha$-regular and
\[ \hat \lh \in L_{\bg}, \quad | \hat \lh | < n^\eta. \]

\begin{lem}
\label{lem:prenormal}
For fixed $(\alpha, \eta) \in \Delta_1$,
\[ \lim_{n \to \infty;\; \hat\bb \in \Xi_{\alpha,\eta} }
\frac{ \sqrt{n} \; P(\hat\bb) }
{ \ga_{\cC} \hat \lh }
= 1. \]
\end{lem}

We must clarify the meaning of the above limit.
In general, let $\Xi$ be a set with a function $n: \Xi \to \NN$.
Let $f: \Xi \to \RR$ be a function and $A \in \RR$. Then
\[ \lim_{n \to \infty;\; \hat\bb \in \Xi } f(\hat\bb) = A \]
means that given $\epsilon > 0$ there exists $n_0$ such that,
if $\hat\bb \in \Xi$ and $n(\hat\bb) > n_0$ then
\[ \left| f(\hat\bb) - A \right| < \epsilon; \]
in the above lemma, the function $n$ is the obvious one.
The set $\Xi_{\alpha,\eta}$ is implicit
(and hopefully clear from context)
whenever we use notations $\sim$ and little oh.
Thus, for instance, the lemma may be stated as
\[ P(\hat\bb) \sim \frac{1}{\sqrt{n}}\,\ga_{\cC} \hat \lh  
\quad \textrm{or} \quad
P(\hat\bb) = (1+o(1)) \frac{1}{\sqrt{n}}\,\ga_{\cC} \hat \lh.
\]



\begin{proof}
The hypotheses imply $\eta < 2\alpha - \frac32$ and therefore
\[ |\hat \lh_{i,j}| < n^\eta
< n^{(2\alpha - \frac32)} < \sqrt{n}\,\hat a_i \hat b_j \]
so that $\hat \cl_{i,j} = \hat a_i \hat b_j (1 + o(1))$.
In particular,
$\lambda(\hat\bb) \to \infty$ when $n \to \infty$
and therefore, by Lemma \ref{lem:stirling}, $P(\hat\bb) \sim \tilde P(\hat\bb)$.
We also have $\cD \sim \cC^2$.
In the notation of Lemma \ref{lem:logQ}, write
\[ \log(P(\hat\bb)) = - \frac{1}{2} \log(2\pi \cC n)
- n A + o(1), \]
\[ A = \sum_{i,j} \hat a_i \hat b_j 
\xl \left( 1 + \frac{\hat \lh_{i,j} }{\sqrt{n} \hat a_i \hat b_j} \right), \quad
\left| \frac{\hat \lh_{i,j} }{\sqrt{n} \hat a_i \hat b_j} \right| < n^{\frac32 + \eta - 2\alpha} = o(1) \]
and we may use the Taylor approximation in Equation \ref{eq:taylorxl}
(with a universal constant in the $O$) to obtain
\[ nA =
\sum_{i,j} \sqrt{n}\,\hat \lh_{i,j} +
\sum_{i,j} \frac{\hat \lh_{i,j}^2}{\hat a_i \hat b_j} +
O\left( \sum_{i,j} n^{-\frac12}
\frac{|\hat \lh_{i,j}|^3}{\hat a_i^2 \hat b_j^2} \right). \]
The first summand equals zero from equation \ref{eq:sumh}.
The second equals $\hat \lh^2/\cC$.
The third summand goes to zero since each of the four terms
is bounded by 
\[ n^{-\frac12} \cdot n^{3\eta} \cdot n^{4(1-\alpha)} =
n^{\frac72 + 3\eta - 4\alpha} = o(1). \]
Thus
\[ \log(P(\hat\bb)) = - \frac{1}{2} \log(2\pi \cC n)
- \frac{\hat \lh^2}{2 \cC} + o(1), \]
completing the proof.
\end{proof}

\begin{cor}
\label{coro:normal}
Let $(u,v) \in (0,1)^2$ be fixed.
For real numbers $t_{-} < t_{+}$, let
\[ p_n = \Pr\left[ 
t_{-}  < \sqrt{n}\;\by_n(u,v) < t_{+} 
\right]. \]
Then
\[ \lim_{n \to \infty} p_n =
\int_{t_{-}}^{t_{+}}
\ga_{\cC}(t) dt. 
\]
\end{cor}

\begin{proof}
Take fixed $(\alpha, \eta) \in \Delta_1$.
Given $n$,
set $a = \lfloor u n \rfloor$, $b = \lfloor v n \rfloor$.
Notice that $|y_n(u,v) - y_n(\frac an, \frac bn)| < 2$.
Let $\bg_n = (n,a,b)$.
For sufficiently large $n$,
the conditions $\hat l \in L_{\bg_n}$,
$t_{-} < \hat l < t_{+}$ imply that
the normalized blocking $\hat\bb_n = (\bg_n,\hat l)$
is $(\alpha,\eta)$-standard.
From Lemma \ref{lem:prenormal},
since the spacing of $L_{\bg_n}$ is $1/\sqrt{n}$, we have
\[
\lim_{n \to \infty} p_n =
\lim_{n \to \infty} \frac{1}{\sqrt{n}}
\sum_{t_{-} < \hat l < t_{+} \atop \hat l \in L_{\bg_n}}
\ga_\cC(\hat l ) \\
=
\int_{t_{-}}^{t_{+}}
\ga_{\cC}(t) dt. 
\]
\end{proof}

The following corollary gives an example of an application
of Lemma \ref{lem:prenormal} for a sequence of blockings $\bb_n$.
The $\sim$ notation below is therefore used in the more traditional sense
of a sequence limit.

\begin{cor}
\label{cor:cacb}
For fixed $(u,v) \in (0,1)^2$,
let $a = \lfloor nu \rfloor$, $b = \lfloor nv \rfloor$,
$\cC = u(1-u)v(1-v)$.
Let $0\leq r<1/6$ and 
$g(n)$ be a function such that $g(n)=o(n^t)$ for any $t>0$.
Then we have
$$\Pr\left[X_{n,a,b}=\left\lfloor nuv + g(n) n^{r+1/2}
\right\rfloor \right] \sim 
\frac{1}{\sqrt{n}} \ga_{\cC}\left( g(n) n^{r} \right).  $$
\end{cor}

\begin{proof}
Take $(\alpha,\eta) \in \Delta_1$ with $r < \eta$.
For sufficiently large $n$, the blocking
\[ \bb_n = \left(n,a,b,\left\lfloor nuv + g(n) n^{r+1/2}
\right\rfloor \right) \]
is $(\alpha,\eta)$-standard.
The result now follows directly from Lemma \ref{lem:prenormal}.
\end{proof}

Lemma \ref{lem:prenormal} above gives a good estimate for
$P(\hat\bb)$ when $\cl$ is relatively near $\frac{ab}{n}$,
or, equivalently, when $\hat l$ is small.
In other words, the probability of $|Y_{n,a,b}|$ (or $|\hat l|$)
being large is small,
but we shall need something more precise.
The following lemma addresses these concerns.
Notice also that Lemma \ref{lem:ineq} below
is not a corollary of Lemma \ref{lem:prenormal}.

\begin{lem}
\label{lem:ineq}
Let $(\alpha, \eta) \in \Delta_1$ be fixed.
Then there exists $n_0$ such that, if $n > n_0$,
for any $\alpha$-regular grid $\bg = (n,a,b)$
and $H \in (1, 2 n^\eta)$,
then
\[ \Pr\left[ | Y_{n,a,b} | > H \sqrt{\cC n} \right] < 
\frac{1}{H} \; \exp\left( - \frac{H^2}{2} \right). \]
\end{lem}


\begin{proof}
We prove that
\[ \Pr\left[ Y_{n,a,b} > H \sqrt{\cC n} \right] <
\frac{1}{2H} \; \exp\left( - \frac{H^2}{2} \right). \]
Let 
\[
p_l =  P(\bg,\cl); \quad
\cl_0 = \left\lfloor \frac{ab}{n} + H \sqrt{\cC n} \right\rfloor;
\]
so that
\[ \tilde p = \Pr\left[X_{n,a,b} > \frac{ab}{n} + H \sqrt{\cC n} \right] =
p_{\cl_0+1} + p_{\cl_0+2} + \cdots + p_{\cl_0+k} + \cdots. \]
The blocking $(\bg,\cl_0+1)$ is $(\alpha,\eta)$-standard
and therefore, from Lemma \ref{lem:prenormal},
\[ p_{\cl_0} <
(1+o(1)) \; \frac{1}{\sqrt{2\pi \cC n}} \;
\exp\left( - \frac{H^2}{2}  \right). 
\]
Let $\lambda_l = p_{l+1}/p_l$.
From Lemma \ref{lem:exact},
\[ \lambda_l = \frac{(a-\cl)(b-\cl)}{(l+1)(n-a-b+\cl+1)}, \]
a decreasing function of $\cl$.
Thus, $p_{\cl_0+k} \le \lambda_{\cl_0}^k p_{\cl_0}$
and therefore
\[ \tilde p < (1+o(1))\; \frac{1}{1-\lambda_{\cl_0}}\;
\frac{1}{\sqrt{2\pi n \cC}} \;
\exp\left( - \frac{H^2}{2}  \right). 
\]
We now proceed to estimate $\lambda_{\cl_0}$.
Set $\hat \cl_0 = \cl_0/n \approx \hat a\hat b + H \sqrt{\cC} n^{-1/2}$.
\begin{align*}
\lambda_{\cl_0} &= \frac{(a-\cl_0)(b-\cl_0)}{(\cl_0+1)(n-a-b+\cl_0+1)} \\
&\approx \frac{(\hat a - \hat \cl_0)(\hat b - \hat \cl_0)}
{\hat \cl_0(1 - \hat a - \hat b + \hat \cl_0)} \\
&\approx \frac{(\hat a - \hat a\hat b - H \sqrt{\cC} n^{-1/2})
(\hat b - \hat a\hat b - H \sqrt{\cC} n^{-1/2})}
{(\hat a\hat b + H \sqrt{\cC} n^{-1/2})
(1 - \hat a - \hat b + \hat a\hat b  + H \sqrt{\cC} n^{-1/2})}.
\end{align*}
The approximations here are due to taking integer parts
and can safely be disconsidered.
Expanding, we have
\begin{align*}
\lambda_{\cl_0} &= \frac{
\cC + (2\hat a\hat b - \hat a - \hat b) H \sqrt{\cC} n^{-1/2} + H^2 \cC n^{-1}}
{\cC + (1+2\hat a\hat b - \hat a - \hat b) H \sqrt{\cC} n^{-1/2} + H^2 \cC n^{-1}}
\\
&= 1 - \left( \frac{H}{\sqrt{\cC n}} (1 + O(n^{-\epsilon})) \right)
\end{align*}
where $\epsilon > 0$.
We therefore have
\[ \frac{1}{1-\lambda_{\cl_0}} =
(1+O(n^{-\epsilon}))\;\frac{\sqrt{\cC n}}{H}, \]
which yields the desired estimate.
A similar argument proves that
\[ \Pr\left[ Y_{n,a,b} < - H \sqrt{\cC n} \right] <
\frac{1}{2H} \; \exp\left( - \frac{H^2}{2} \right), \]
completing the proof.
\end{proof}

\section{Grids and joint distributions}
\label{sect:grid}

In this section we generalize some of the results
of Section \ref{sect:pointwise} to \emph{grids}.

Recall that a \textit{bridged Brownian sheet} is a process
yielding continuous functions $f \in C^0_0(\II^2)$.
We provide a characterization
which will be appropriate to our proof.
Consider a $\mi \times \mj$ \emph{grid} $\bg$
of points $(u_i,v_j) \in \II^2$ where
\[ 0 = u_0 < u_1 < \cdots < u_{\mi} = 1, \quad
0 = v_0 < v_1 < \cdots < v_{\mj} = 1. \]
In order to generate the values $f\left(u_i,v_j\right)$
of a bridged Brownian sheet on the grid,
take independent and normally distributed numbers $Z_{i,j}$,
$1 \le i \le \mi$, $1 \le j \le \mj$,
with average $0$ and variance $1$, to obtain a vector $z \in \RR^{\mi\mj}$.
Let $V_\bg \subset \RR^{\mi\mj}$ be
the subspace of dimension $(\mi - 1)(\mj - 1)$ defined by
\[ \sum_i \sqrt{u_i - u_{i-1}}\, z_{i,j} = 0, \quad
\sum_j \sqrt{v_j - v_{j-1}}\, z_{i,j} = 0. \]
Project $z \in \RR^{\mi\mj}$ orthogonally to $\tilde z \in V_\bg$.
Set
\begin{equation}
\label{eq:jointbrown}
f\left(u_i,v_j\right) =
\sum_{i'\le i; \; j'\le j}
\sqrt{(u_{i'} - u_{i'-1})(v_{j'} - v_{j'-1})}\;\tilde z_{i',j'}.
\end{equation}
Thus, the values of $f$ on a fixed set of points follow
a joint normal distribution.
In particular, for fixed $(u,v) \in \II^2$,
the random variable $f(u,v)$
has normal distribution with average $0$ and variance
$u(1-u)v(1-v)$.
In Section \ref{sect:largegrid}  we  focus
on square grids $\mi = \mj = m$,
$u_i = \frac{i}{m}$, $v_j = \frac{j}{m}$.

\medskip

This section is dedicated to the following result.

\begin{thm}
\label{theo:joint}
Given $(u_1,v_1), \ldots (u_m,v_m) \in (0,1)^2$,
the joint distribution of $\sqrt{n}\,\by_n(u_i,v_i)$
converges (when $n \to +\infty$)
to the joint distribution of $f(u_i,v_i)$ for a 
bridged Brownian sheet $f$.
\end{thm}

Notice that the case $m = 1$ of this theorem
follows from Section \ref{sect:pointwise}.

As before, $n$ is a positive integer
and $X$ is a random discrete copula with uniform distribution.
A $\mi \times \mj$ \textit{grid} is a pair $\bg = (a,b)$ of families,
$(a_i)_{0 \le i \le \mi}$ and $(b_j)_{0 \le j \le \mj}$
of integers with
\[ 0 = a_0 < \cdots < a_i < \cdots < a_{\mi} = n, \quad
0 = b_0 < \cdots < b_j < \cdots < b_{\mj} = n. \]
As in Figure \ref{fig:grid},
the values of $a_i$ and $b_j$ should be interpreted as
horizontal and vertical lines dividing a matrix $M$
into \textit{boxes} of dimensions
$\dot a_i \times \dot b_j$ defined by
\[
\dot a_i = a_i - a_{i-1}, \quad
a_{i_0} = \sum_{i \le i_0} \dot a_i, \quad
\dot b_j = b_j - b_{j-1}, \quad
b_{j_0} = \sum_{j \le j_0} \dot b_j.
\]
Given a discrete copula $C$ and a grid $\bg = (a,b)$,
the number of $1$'s (of $M$) in box $(i,j)$ is
$C_{a_i,b_j} - C_{a_{i-1},b_j} - C_{a_i,b_{j-1}} + C_{a_{i-1},b_{j-1}}$;
we are interested in the statistics of such numbers.


\begin{figure}[ht]
\centering
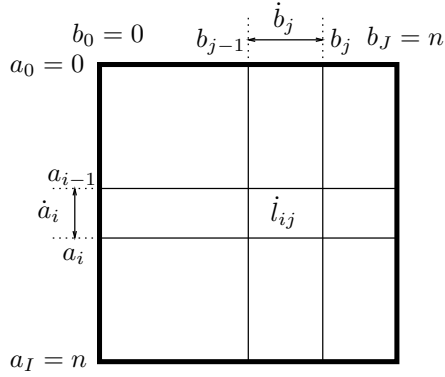
\caption{The box $(i,j)$}
\label{fig:grid}
\end{figure}

Let $\cl = (\cl_{i,j})$ be a family of integers with
\[ i_0 \le i_1, j_0 \le j_1 \Rightarrow
\cl_{i_0,j_0} + \cl_{i_1,j_1} \ge \cl_{i_0,j_1} + \cl_{i_1,j_0}, \]
\[ \cl_{i,0} = \cl_{0,j} = 0, \quad \cl_{i,\mj} = a_i, \quad \cl_{\mi,j} = b_j. \]
Define
\[ \dot \cl_{i,j} = \cl_{i,j} - \cl_{i-1,j} - \cl_{i,j-1} + \cl_{i-1,j-1}, \quad
\cl_{i_0,j_0} = \sum_{i \le i_0 \atop j \le j_0} \dot \cl_{i,j}. \]
For a blocking $\bb = (\bg,\cl)$, we are interested in estimating
\[ P(\bb) =
\Pr\left[ \forall i, \forall j, X_{a_i,b_j} = \cl_{i,j} \right]; \]
this is the probability that each box $(i,j)$
contains precisely $\dot \cl_{i,j}$ $1$'s.
The situation described in Section \ref{sect:pointwise}
corresponds to $\mi = \mj = 2$
(the old $a$ is the new $a_1$; the old $\cl$ is the new $\cl_{1,1}$).
A careful estimate of $P(\bb)$ will prove Theorem \ref{theo:joint}
for the special case of points in a grid.
This special case implies the general one:
indeed, it suffices to add points in order to complete a grid.

First a simple remark which shall be used several times:
there is nothing special about the upper left corner, 
and we can count nonzero entries is any rectangular submatrix.

\begin{lem}
\label{lem:Delta}
Let $0 < a_1 < a_2 < n$, $0 < b_1 < b_2 < n$ and $\cl$ be integers.
Then
\[ \Pr\left[
X_{a_2,b_2} - X_{a_2,b_1} - X_{a_1,b_2} + X_{a_1,b_1} = \cl \right] =
P(n,a_2-a_1,b_2-b_1,\cl). \]
\end{lem}

\begin{proof}
Cyclically permute rows and columns of the corresponding matrix $M$.
\end{proof}

The following lemma generalizes Lemma \ref{lem:exact}.

\begin{lem}
\label{lem:moreexact}
\[ P(\bb) =
\frac{
\left( \prod_{0 < i \le \mi} \dot a_i! \right)
\left( \prod_{0 < j \le \mj} \dot b_j! \right)
}{
n!
\left( \prod_{0 < i \le \mi \atop 0 < j \le \mj} \dot \cl_{i,j}! \right)
}.
\]
\end{lem}

\begin{proof}
We consider the corresponding permutation matrix by columns:
these are the vectors $e_1, \ldots, e_n$,
and we have to choose their order.
There are $\dot a_i$ vectors of \emph{type $i$}:
$e_{a_{i-1}+1}, \ldots, e_{a_i}$.
The \emph{wide column} $j$
is the union of the $\dot b_j$ columns $b_{j-1}+1$ to $b_j$;
there are therefore $\mj$ wide columns.
We have to place $\dot \cl_{i,j}$
vectors of type $i$ in the wide column $j$.
Thus, for each $i$, we have to choose $\dot \cl_{i,j}$
vectors to go inside each wide column: there are
\( \left( {\prod_i \dot a_i!} \right)/
\left( {\prod_{i,j} \dot \cl_{i,j}!} \right) \)
ways of choosing.
Inside each wide column, vectors may be permuted,
obtaining the desired count of suitable permutation matrices.
\end{proof}

Rescale to introduce
\[ \hat a_i = \dot a_i/n \in \II, \quad
\hat b_j = \dot b_j/n \in \II, \quad
\hat \cl_{i,j} = \dot \cl_{i,j}/n \in \II, \]
\[ \mathcal{C}_a = \prod_i \hat a_i, \quad
\mathcal{C}_b = \prod_j \hat b_j, \quad
\mathcal{D}= \prod_{i,j} \hat \cl_{i,j} \]
and the factorial-free approximation
\[ \tilde P(\bg,\cl) =
\sqrt{\frac{\cC_a \cC_b}{(2\pi n)^{(\mi-1)(\mj-1)} \cD}} \;
\frac{ \left( \prod_i \hat a_i^{(\hat a_i n)} \right)
\left( \prod_j \hat b_j^{(\hat b_j n)} \right) }
{ \prod_{i,j} \hat \cl_{i,j}^{(\hat \cl_{i,j} n)} }, \]
or, equivalently, with $\xl(t) = t \log t$,
\begin{gather*}
\log(\tilde P(\bg,\cl)) = 
-\frac{(\mi-1)(\mj-1)}{2}\log\left(2\pi n\right)
- \frac{1}{2} \log\left(\frac{\cC_a \cC_b}{\cD}\right) \\
+ n \left( \sum_i \xl(\hat a_i) + \sum_j \xl(\hat b_j)
- \sum_{i,j} \xl(\hat \cl_{i,j}) \right). 
\end{gather*}

The following lemma generalizes Lemma \ref{lem:stirling}.
The \emph{sparsity} of a blocking $\bb = (\bg,\cl)$ is
$\lambda(\bb) = \min_{i,j} \dot \cl_{i,j}$.

\begin{lem}
\label{lem:morestirling}
Let $\mi$ and $\mj$ be fixed.
Then, when $\lambda(\bb) \to \infty$,
$P(\bb) \sim \tilde P(\bb)$.
\end{lem}

\begin{proof}
Use Stirling's formula on Lemma \ref{lem:moreexact}:
\begin{gather*}
\dot a_i! \sim (n/e)^{\hat a_i n} \hat a_i^{\hat a_i n} 
\sqrt{\hat a_i} \sqrt{2\pi n}, \quad
\dot b_j! \sim (n/e)^{\hat b_j n} \hat b_j^{\hat b_j n} 
\sqrt{\hat b_j} \sqrt{2\pi n}, \\
n! \sim (n/e)^n \sqrt{2\pi n}, \quad
\dot \cl_{i,j}! \sim (n/e)^{\hat \cl_{i,j} n} \hat \cl_{i,j}^{\hat \cl_{i,j} n} 
\sqrt{\hat \cl_{i,j}} \sqrt{2\pi n}. 
\end{gather*}
We have $\sum_i \hat a_i = \sum_j \hat b_j = \sum_{i,j} \hat \cl_{i,j} = 1$
and therefore
\[ \frac{ \left( \prod_i (n/e)^{\hat a_i n} \right)
\left( \prod_j (n/e)^{\hat b_j n} \right) }
{ (n/e)^n \left( \prod_{i,j} (n/e)^{\hat \cl_{i,j} n} \right) } = 1. \]
The result follows.
\end{proof}


At this point in Section \ref{sect:pointwise}
we introduce the variable $\hat \lh$ and random variable $\hat h$.
Here we have reason to introduce slightly different variables
$\tilde \lh_{i,j}$ (and corresponding variables $\tilde h_{i,j}$)
defined by
\begin{equation}
\label{eq:tildeh}
\hat \cl_{i,j} = \hat a_i \hat b_j +
\sqrt{\frac{\hat a_i \hat b_j}{n}}\; \tilde \lh_{i,j} =
\hat a_i \hat b_j \left( 1 +
\frac{\tilde \lh_{i,j}}{\sqrt{n \hat a_i \hat b_j}} \right).
\end{equation}
We relate a blocking $\bb = (\bg,\cl)$ with a \emph{normalized blocking}
$\tilde\bb = (\bg,\tilde \lh)$.
Notice that $\tilde \lh$ belongs to the vector subspace
$V_\bg \subset \RR^{\mi\mj}$ of dimension $(\mi - 1)(\mj - 1)$
defined by
\begin{equation}
\label{eq:Vg}
\sum_i \sqrt{\hat a_i}\, \tilde \lh_{i,j} = 0, \quad
\sum_j \sqrt{\hat b_j}\, \tilde \lh_{i,j} = 0.
\end{equation}
Furthermore, $\tilde \lh$ belongs to the lattice $L_{\bg} \subset V_{\bg}$
defined by
\begin{equation}
\label{eq:Lg}
\tilde \lh_{i,j} \in \frac{1}{\sqrt{n \hat a_i \hat b_j}}
\left( \ZZ - \frac{\dot a_i \dot b_j}{n} \right).
\end{equation}
Thus, the random variable $\tilde h$ assumes values in $L_{\bg}$.

\begin{lem}
\label{lem:morelogQ}
\begin{align*}
\log(\tilde P(\bg,\tilde \lh)) &=
\frac{1}{2} \log\left(\frac{\cC_a \cC_b}{\cD}\right)
- \frac{(\mi-1)(\mj-1)}{2} \log(2\pi n) \\
&\phantom{=} - n \sum_{i,j} \hat a_i \hat b_j
\xl\left( 1 + \frac{\tilde \lh_{i,j}}{\sqrt{n\,\hat a_i \hat b_j}} \right). 
\end{align*}
\end{lem}

\begin{proof}
Up to notation, like the proof of Lemma \ref{lem:logQ}.
\end{proof}

We now extend Lemma \ref{lem:prenormal} to grids.
As in Section \ref{sect:pointwise},
let $\alpha \in (\frac{11}{12},1)$; a grid $\bg$ is
\textit{$\alpha$-regular} if
\[ \forall (i,j) \; \dot a_i, \dot b_j > n^\alpha. \]
Consider the open triangle 
\begin{equation}
\label{eq:Delta2}
\Delta_2 = \{ (\alpha,\eta) \in \RR^2 \;\mid\;
0 < 12 \eta < 12\alpha - 11 < 1 \} \subset \Delta_1.
\end{equation}
For $(\alpha,\eta) \in \Delta_2$,
a normalized blocking $\tilde\bb = (\bg,\tilde \lh)$
is \textit{$(\alpha,\eta)$-standard}
(denoted by $\tilde\bb \in \Xi_{\alpha,\eta}$)
if $\bg$ is $\alpha$-regular, $\tilde \lh \in L_{\bg}$, and
\[ \forall (i,j) \; | \tilde \lh_{i,j} | < n^\eta. \]

\begin{lem}
\label{lem:moreprenormal}
Fix $\mi$, $\mj$ and $(\alpha,\eta) \in \Delta_2$.
If we restrict ourselves to $\Xi_{\alpha,\eta}$ then,
as $\lambda(\bb)$ goes to infinity,
\[ P(\bg,\tilde \lh) \sim \frac{1}{%
\sqrt{(2\pi n)^{(\mi-1)(\mj-1)} \cC_a^{(\mj-1)} \cC_b^{(\mi-1)}}}\;
\exp\left( - \frac{|\tilde \lh|^2}{2} \right).
\]
\end{lem}

Here $|\tilde\lh|^2 = \sum_{i,j} \tilde \lh_{i,j}^2$.
The notation $\sim$ here means the same as discussed in detail
between the statement and proof of Lemma \ref{lem:prenormal}.
Notice that when $\lambda(\bb)$ goes to infinity, so does $n$.
We shall prove in Lemma \ref{lem:goodh} below
that, for $\alpha$-regular grids,
the condition $|\tilde h_{i,j}| < n^\eta$ holds
with probability $1 - o(1)$ (when $n$ goes to infinity).

\begin{proof}
Notice that the conditions imply
\[ \frac{\tilde \lh_{i,j}}{\sqrt{n \hat a_i \hat b_j}}
< n^{\eta - \alpha + \frac12}  = O(n^{-\frac{5}{12}});
\quad
\frac{\tilde \lh_{i,j}^3}{\sqrt{n \hat a_i \hat b_j}}
< n^{3\eta - \alpha + \frac12} = O(n^{-\frac{1}{4}}). \]
Let
\[
\tilde P_0(\bg,\cl) = 
\sqrt{\frac{\cC_a \cC_b}{(2\pi n)^{(\mi-1)(\mj-1)} \cD}}, \quad
\tilde P_1(\bg,\cl) = \frac{ \left( \prod_i \hat a_i^{(\hat a_i n)} \right)
\left( \prod_j \hat b_j^{(\hat b_j n)} \right) }
{ \prod_{i,j} \hat \cl_{i,j}^{(\hat \cl_{i,j} n)} }. 
\]
Lemma \ref{lem:morestirling} gives the uniform estimate
$P(\bg,\cl) \sim \tilde P_0(\bg,\cl) \tilde P_1(\bg,\cl)$. 

We first estimate $\tilde P_0$.
As before,
\[ \hat \cl_{i,j} 
= \hat a_i\hat b_j
\left(1 + \frac{\tilde \lh_{i,j}}{\sqrt{n\hat a_i\hat b_j}}\right)
= \hat a_i\hat b_j\left(1 + O(n^{-\frac{5}{12}}) \right). \]
We have that $\hat \cl_{i,j} \sim \hat a_i\hat b_j$ (uniformly) and therefore
\[ \cD = \prod_{i,j} \hat \cl_{i,j} \sim \prod_{i,j} (\hat a_i \hat b_j) =
\left( \prod_i \hat a_i \right)^{\mj}
\left( \prod_j \hat b_j \right)^{\mi}
= \cC_a^{\mj} \cC_b^{\mi}. \]
Thus
\[  \tilde P_0(\bg,\cl) \sim \frac{1}{%
\sqrt{(2\pi n)^{(\mi-1)(\mj-1)} \cC_a^{(\mj-1)} \cC_b^{(\mi-1)}}}. \]
We now consider $\tilde P_1$.
From Lemma \ref{lem:morelogQ},
using the Taylor approximation for $\xl(1+t)$
(Equation \ref{eq:taylorxl}),
\begin{align*}
\log\left( \tilde P_1(\bg,\cl) \right) &= 
- \sum_{i,j} (n \hat a_i \hat b_j)
\xl\left(1+\frac{\tilde \lh_{i,j}}{\sqrt{n \hat a_i \hat b_j}} \right) \\
&= -\sqrt{n} \sum_{i,j} \sqrt{\hat a_i \hat b_j} \tilde \lh_{i,j} 
- \sum_{i,j} \frac{\tilde \lh_{i,j}^2}{2} + O(n^{-\frac14})
\end{align*}
or, since $\sum_{i,j} \sqrt{\hat a_i \hat b_j} \tilde \lh_{i,j} =
\sum_i \sqrt{\hat a_i} \left( \sum_j \sqrt{\hat b_j} \tilde \lh_{i,j} \right)
= 0$, 
\[ \tilde P_1(\bg,\cl) \sim \exp\left(
- \sum_{i,j} \frac{\tilde \lh_{i,j}^2}{2} 
\right), \]
completing the proof.
\end{proof}

\begin{lem}
\label{lem:goodh}
Fix $\mi$, $\mj$ and $(\alpha,\eta) \in \Delta_2$.
Then there exists $n_0$ such that,
if $n > n_0$ and if $\bg$ is an $\alpha$-regular $\mi \times \mj$ grid,
then
\[ \Pr\left[\forall(i,j)\; |\tilde h_{i,j}| < n^\eta \right] >
1 - \frac{\mi \mj}{n^\eta} \exp\left( - \frac{n^{2\eta}}{2} \right)
= 1 - o(1). \]
\end{lem}

\begin{proof}
The idea is to use Lemma \ref{lem:ineq} with $H = n^\eta$.
Fix $i$ and $j$.
Notice that
\[ \tilde h_{i,j} = \frac{Y_{n,a_i,b_j}}{\sqrt{n\,\hat a_i\,\hat b_j}};
\quad \cC = \hat a_i(1-\hat a_i)\hat b_j(1-\hat b_j). \]
We have
\begin{align*}
\Pr\left[ |\tilde h_{i,j}| > n^\eta \right] &=
\Pr\left[ |Y_{n,a_i,b_j}| > H \sqrt{n\,\hat a_i\,\hat b_j} \right] \\
& \le \Pr\left[ |Y_{n,a_i,b_j}| > H \sqrt{\cC\,n} \right] <
\frac{1}{H} \exp\left( - \frac{H^2}{2} \right).
\end{align*}
Adding up for the $IJ$ pairs we complete the proof.
\end{proof}

\begin{proof}[Proof of Theorem \ref{theo:joint}]
Fix $(\alpha,\eta) \in \Delta_2$.
By projection onto an appropriate subspace,
if Theorem \ref{theo:joint} holds for
the finite family $(u_1,v_1),\ldots,(u_m,v_m)$
it also holds for any subfamily.
Assume therefore without loss that the family forms
a grid $\bg_0$ of points $(u_i,v_j)$, with $i < \mi$, $j < \mj$.
Write $u_0 = v_0 = 0$, $u_{\mi} = v_{\mj} = 1$.

For a positive integer $n$, set
$a_i = \lfloor n u_i \rfloor$ and
$b_j = \lfloor n v_j \rfloor$.
For large $n$, this defines an $\alpha$-regular grid $\bg$ with
$\hat a_i \approx u_i - u_{i-1}$ and $\hat b_j \approx v_j - v_{j-1}$.
Thus $V_{\bg} \approx V_{\bg_0}$.
The bridged Brownian sheet defines a normal distribution on $V_{\bg_0}$
as in Equation \ref{eq:jointbrown}.
Lemma \ref{lem:goodh} shows that, with probability $1 - o(1)$,
a random permutation obtains an $(\alpha,\eta)$-standard
normalized blocking $(\bg,\tilde h)$.
From Lemma \ref{lem:moreprenormal},
such $\tilde h$ follow a normal distribution in the lattice
$L_{\bg} \subset V_{\bg}$.
Since these lattices get finer when $n$ goes to infinity 
it follows that the distributions of $\sqrt{n}\,\by_n$
converge to that of $f$, completing the proof.
\end{proof}

\section{H\"older estimates}
\label{sect:holder}

Theorem \ref{theo:discrete} claims that $\by_n$ approaches
a bridged Brownian sheet. As is well known, Brownian sheets
are (almost surely) H\"older continuous for any exponent
smaller than $1/2$ (bridged or not).
The following lemma is a formal version of this fact
and will be needed in the proof of Theorem \ref{theo:discrete}.
We use the following technical definition:
given $r, \delta, C > 0$,
a continuous function $g: \II^2 \to \RR$ is
$(r,\delta,C)$-H\"older if,
for all $(u_0,v_0), (u_1,v_1) \in \II^2$,
\[ |u_0 - u_1|, |v_0 - v_1|  < r \; \Rightarrow \;
|f(u_0,v_0) - f(u_1,v_1)| <
C ( |u_0 - u_1| + |v_0 - v_1| )^{\frac12 - \delta}.
\]

\begin{lem}
\label{lem:holderbrownsheet}
Let $f: \II^2 \to \RR$ be a bridged Brownian sheet.
Fix $\delta, C > 0$.
Let $p_r$ be the probability that $f$ be $(r,\delta,C)$-H\"older.
Then
\[ \lim_{r\searrow 0} p_r = 1. \]
\end{lem}

The proof is analogous to the well known
case of Brownian motion and is omitted.

Lemma \ref{lem:holder} below obtains a related estimate for $Y$,
to be used in the proof of Theorem \ref{theo:discrete}.
A discrete copula $C$ is \emph{$(\alpha,\epsilon)$-H\"older} if
\[ \forall (a_1, a_2, b_1, b_2),\;
\left( |a_2 - a_1|, |b_2 - b_1| < n^\alpha
\; \Rightarrow \;
\left| D_{a_2,b_2} - D_{a_1,b_1} \right|
< n^{\frac{\alpha}{2} +\epsilon} \right) \]
(where, as usual, $D_{a,b} = C_{a,b} - \frac{ab}{n}$).

\begin{lem}
\label{lem:holder}
Let $\alpha \in (\frac78,1)$ and $\epsilon > 0$ be fixed.
Then, a.a.s., $X$ is $(\alpha,\epsilon)$-H\"older.
\end{lem}

In other words, 
\[ \lim_{n \to \infty}
\Pr\left[ X \textrm{ is $(\alpha,\epsilon)$-H\"older} \right] = 1. \]

\begin{proof}
Notice first that it suffices to consider the case $b_1 = b_2$.
That is, we prove that a.a.s.,
for all $a_1, a_2, b$,
if $|a_1 - a_2| < n^\alpha$
then $|Y_{a_1,b} - Y_{a_2,b}| < n^{\frac{\alpha}{2} + \epsilon}$.
Indeed, symmetry then implies a similar result
in the case $a_1 = a_2$.
But then, a.a.s., 
\begin{gather*}
\left| Y_{a_2,b_2}   - Y_{a_1,b_1}   \right| \le 
\left| Y_{a_2,b_2}   - Y_{a_2,b_1}   \right| +
\left| Y_{a_2,b_1}   - Y_{a_1,b_1}   \right| < 
2 n^{\frac{\alpha}{2} + \epsilon};
\end{gather*}
changing $\epsilon$ obtains the desired result.

Fix $\eta$ such that $(\alpha,\eta) \in \Delta_1$
(as defined in Equation \ref{eq:Delta1}).
We may assume that $\epsilon < \eta/4$.
We first consider $a_1 < a_2$ and $b$ fixed.
From Lemma \ref{lem:Delta} we have
\[ \Pr\left[
\left| Y_{a_2,b} - Y_{a_1,b} \right|
> n^{\frac{\alpha}{2} +\epsilon} \right] =
\Pr\left[
\left| Y_{a_2-a_1,b} \right|
> n^{\frac{\alpha}{2} +\epsilon} \right]. \]
Assume first that $n^\alpha < b < n - n^\alpha$.
Let $a_3 = a_2 - a_1 + \lfloor n^\alpha \rfloor + 1$
so that $n^\alpha < a_3 < 2 n^\alpha$.
We apply Lemma \ref{lem:ineq}
with $a = a_3$, $b = b$ and $H = n^{\epsilon/2}$:
here $\cC < 2 n^{\alpha-1}$ and
$H \sqrt{\cC n} < \frac{1}{4} n^{\frac{\alpha}{2}+\epsilon}$.
We have
\begin{gather*}
\Pr\left[ \left| Y_{a_3,b} \right|
> \frac{1}{4} n^{\frac{\alpha}{2} +\epsilon} \right]
< \Pr\left[ \left| Y_{a_3,b} \right|
> H \sqrt{\cC n} \right] < \\
< n^{-\epsilon/2} \exp\left(-\frac{n^{\epsilon}}{2}\right)
< \exp(-n^{\epsilon/2}).
\end{gather*}
Again from Lemma \ref{lem:Delta},
\[ \Pr\left[ \left| Y_{a_3,b} - Y_{a_2-a_1,b} \right|
> \frac{1}{4} n^{\frac{\alpha}{2} +\epsilon} \right]
= \Pr\left[ \left| Y_{a_3-(a_2-a_1),b} \right|
> \frac{1}{4} n^{\frac{\alpha}{2} +\epsilon} \right]; \]
a computation similar to the previous one gives
\[ \Pr\left[ \left| Y_{a_3,b} - Y_{a_2-a_1,b} \right|
> \frac{1}{4} n^{\frac{\alpha}{2} +\epsilon} \right]
< \exp(-n^{\epsilon/2})
\]
and therefore
\[ \Pr\left[ \left| Y_{a_2-a_1,b} \right| 
> \frac{1}{2} n^{\frac{\alpha}{2} +\epsilon} \right]
< 2 \exp(-n^{\epsilon/2}). \]
Consider now $b \le n^\alpha$;
let $b_3 = b + \lfloor n^\alpha \rfloor + 1$.
From our previous results,
\[ \Pr\left[ \left| Y_{a_2-a_1,b_3} \right| 
> \frac{1}{2} n^{\frac{\alpha}{2} +\epsilon} \right]
< 2 \exp(-n^{\epsilon/2}); \]
\[ \Pr\left[ \left| Y_{a_2-a_1,b_3-b} \right| 
> \frac{1}{2} n^{\frac{\alpha}{2} +\epsilon} \right]
< 2 \exp(-n^{\epsilon/2}). \]
But again from Lemma \ref{lem:Delta},
\[ \Pr\left[ \left| Y_{a_2-a_1,b_3-b} \right| 
> \frac{1}{2} n^{\frac{\alpha}{2} +\epsilon} \right] =
\Pr\left[ \left| Y_{a_2-a_1,b_3} - Y_{a_2-a_1,b} \right| 
> \frac{1}{2} n^{\frac{\alpha}{2} +\epsilon} \right], \]
and we have
\[ \Pr\left[ \left| Y_{a_2-a_1,b} \right| 
> n^{\frac{\alpha}{2} +\epsilon} \right]
< 4 \exp(-n^{\epsilon/2}). \]
The case $b \ge n - n^\alpha$ is similar.

We now consider all valid triples $(a_1,a_2,b)$:
there are fewer that $n^3$ such triples.
Thus
\[ \Pr\left[ \exists (a_1,a_2,b), |a_1 - a_2| < n^\alpha,
\;\left| Y_{a_2-a_1,b} \right| 
> n^{\frac{\alpha}{2} +\epsilon} \right]
< 4 n^3 \exp(-n^{\epsilon/2}) \to 0, \]
completing the proof.
\end{proof}

The following result can be seen as a concentration of measure:
almost every $X$ is near the product copula $C_0$.
The exponent $1/2$ is sharp;
we shall prove a stronger result later.

\begin{prop}
Let $X$ be a random discrete copula with uniform distribution.
Then, a.a.s.,
\[ \forall (a,b),\; \left|Y_{n,a,b}\right| <  \sqrt{n\log(n)}. \]
\end{prop}


\begin{proof}
Consider $\gamma \in (1/2,1)$.
Consider a subset $M$ of $\{0, 1, \ldots, n\}$
with $m \approx n^{(1-\gamma)}$ elements and gaps
between consecutive elements at a bounded distance from $n^{\gamma}$.
With $\alpha = \gamma$,
given an arbitrary pair $(a_1,b_1)$ there exists $(a_2,b_2) \in M^2$
with $|a_1-a_2|, |b_1-b_2| < n^\alpha$ and therefore,
from Lemma \ref{lem:holder} (a.a.s.),
\[ |Y_{a_1,b_1} - Y_{a_2,b_2}| < n^{(1+\gamma)/4}. \]

Set $\alpha = (2+\gamma)/3$, $\eta = (1-\gamma)/5$
and $H = 2 \sqrt{\log(n)}$.
For $(a, b) \in M^2$ we have $\cC \le 1/16$
so that $H \sqrt{\cC n} < \frac{1}{2} \sqrt{n \log(n)}$.
From Lemma \ref{lem:ineq},
\[ \Pr\left[ \left|Y_{n,a,b}\right|> \frac{1}{2} \sqrt{n\log(n)}\right] <
\Pr\left[ \left|Y_{n,a,b}\right|> H \sqrt{\cC n} \right] <
\frac{1}{2 \sqrt{\log(n)}}\; n^{-2}. \]
We now have
\begin{gather*}
\Pr\left[ \left( \max_{(a,b) \in M^2} \left|Y_{n,a,b}\right| \right)
> \frac{1}{2} \sqrt{n\log(n)}\right]  \le \\
\le \sum_{ (a, b) \in M^2 }
\Pr\left[ \left|Y_{n,a,b}\right|> \frac{1}{2} \sqrt{n\log(n)}\right]
\le \frac{1}{2 \sqrt{\log(n)}}\;m^2 n^{-2} \to 0.
\end{gather*}
Thus, a.a.s.,
\[ \max_{(a,b)} \left|Y_{n,a,b}\right|
\le \frac{1}{2} \sqrt{n\log(n)} + n^{(1+\gamma)/4}
<  \sqrt{n\log(n)}, \]
as desired.
\end{proof}

The weak point in the above proof is handling all points in $M^2$
independently, which we know to be far from true.
In the next section we study joint distributions
in order to improve this result and our understanding of $X$ and $Y$.


\section{Large grids}
\label{sect:largegrid}

In order to complete the proof of Theorem \ref{theo:discrete},
we now consider special grids with $\mi = \mj$
growing as a function of $n$.
Given $\gamma \in (5/6,1)$,
we construct the \textit{$\gamma$-grid} $\bg_\gamma$.
Let $m = \mi = \mj = \lfloor n^{(1-\gamma)} \rfloor$;
for $0 \le i \le n$, let $a_i = b_i = \lfloor \frac{in}{m} + \frac12 \rfloor$
so that gaps are approximately equal:
$\dot a_i = a_i - a_{i-1} \in (n^\gamma - 1, n^\gamma+1)$.

We already established a correspondence between families
$(\cl_{i,j})$ and $(\dot \cl_{i,j})$.
Recall that the same information is given by the family $(\tilde\lh_{i,j})$
of real numbers with $\tilde\lh \in L_{\bg_\gamma} \subset V_{\bg_\gamma}$.
Notice that
$\sum_i \tilde\lh_{i,j} \approx \sum_j \tilde\lh_{i,j} \approx 0$
and that
\[ \dot \cl_{i,j} \approx
n^{(2\gamma - 1)} + \tilde\lh_{i,j} n^{(\gamma-\frac{1}{2})} =
n^{(2\gamma - 1)}(1 + \tilde\lh_{i,j} n^{(-\gamma+\frac{1}{2})}); \]
the small errors arise from taking integer parts in the definition
of $a_i$, $b_j$.
As usual, we write $P(\bg_\gamma;\tilde\lh) = P(n,a,b,\cl)$.

In Section \ref{sect:grid}, we needed Lemma \ref{lem:goodh}
in order to prove Theorem \ref{theo:joint}.
We follow a similar strategy here, and the analogue of
Lemma \ref{lem:goodh} is Lemma \ref{lem:begood} below.
Let $L_{\bg_\gamma,\eta} \subset L_{\bg_\gamma}$ be the finite set of vectors
$\tilde l$ such that, for all $(i,j)$,
we have $|\tilde\lh_{i,j}| < n^{\eta}$.

\begin{lem}
\label{lem:begood}
Consider $\gamma \in (5/6,1)$ and $\eta > 0$ fixed; then
\[ \lim_{n \to \infty} \Pr[\tilde h \in L_{\bg_\gamma,\eta}] = 1. \]
\end{lem}

\begin{proof}
We will use Lemma \ref{lem:ineq} with $\alpha$
slightly smaller than our $\gamma$.
We have $n^\alpha < a_1 = b_1 < n - n^\alpha$
and $|a_1 - n^\gamma| < 1$ and therefore
\[ \left| \frac{a_1 b_1}{n} - n^{2\gamma - 1} \right|
< 2 n^{\gamma-1} + n^{-1} < 1. \]
We have
\[ |\tilde  h_{1,1}|  n^{(\gamma-\frac{1}{2})} =
|X_{a_1,b_1} - n^{(2\gamma - 1)}|
\le |Y_{a_1,b_1}| +
\left| \frac{a_1 b_1}{n} - n^{(2\gamma - 1)} \right|
< |Y_{a_1,b_1}| + 1. \]
Also, $\sqrt{\cC n} < n^{(\gamma - \frac12)}$.
Set $H = \frac12 n^\eta$: from Lemma \ref{lem:ineq},
\[ \Pr\left[ | \tilde  h_{1,1} | \ge n^\eta \right] 
\le \Pr\left[ |  Y_{a_1,b_1} | > \sqrt{\cC n} H \right]
< 2 n^{-\eta} \exp\left( - \frac{n^{2\eta}}{8} \right). \]
From Lemma \ref{lem:Delta} we know that the other cases are similar,
that is, 
\[ \Pr\left[ | \tilde  h_{i,j} | \ge n^\eta \right] 
< 2 n^{-\eta} \exp\left( - \frac{n^{2\eta}}{8} \right). \]
for all $(i,j)$ and therefore
\[ \Pr\left[ \exists (i,j),\;| \tilde  h_{i,j} | \ge n^\eta \right]
< 2 n^{2-\eta} \exp\left( - \frac{n^{2\eta}}{8} \right)
\to 0,  \]
completing the proof.
\end{proof}

We present a result similar to Lemmas \ref{lem:logQ} and \ref{lem:morelogQ}.
For simplicity, from now on we take $\gamma = 19/20$ and $\eta = 1/20$.

\begin{lem}
\label{lem:goodnormal}
Set $\gamma = 19/20$ and $\eta = 1/20$.
For $n \in \NN$, consider the $\gamma$-grid $\bg_\gamma$.
If $\tilde\lh \in L_{\bg_\gamma,\eta}$ then
\begin{align*}
\log(P(\bg_\gamma,\tilde\lh)) &= C_n
 - \frac{|\tilde\lh|^2}{2} + o(1), \\
C_n &= 
- \frac{(m-1)^2}{2} \log(2\pi)
+ \left( \frac{m^2-1}{2} - m(m-1)\gamma \right) \log n.
\end{align*}
\end{lem}

Notice that the little oh notation is employed in the sense discussed
after the statement of Lemma \ref{lem:prenormal}.

\begin{proof}
We know from Lemma \ref{lem:moreexact} that
\[ \log(P(\bg_\gamma,\tilde\lh)) =
\sum_i \log(\dot a_i!) 
+ \sum_j \log(\dot b_j!) 
- \log(n!)
- \sum_{i,j} \log(\dot \cl_{i,j}!); \]
the idea is again to use Stirling approximation
but since the number of terms is unbounded we must proceed
more carefully.
Expand using Equation \ref{eq:stirling}, 
\[ \log(k!) = \xl(k) - k + \frac12 \log(k) + \frac12 \log(2\pi) + \st(k),
\quad 0 < \st(k) < 1/(12n),  \]
and consider the five such terms from right to left.
\begin{gather*}
|T_5| = \left| \sum_i \st(\dot a_i) + \sum_j \st(\dot b_j)
- \st(n) - \sum_{i,j} \st(\dot \cl_{i,j}) \right| < \\
< \frac{1}{12} \left(
\sum_i \frac{1}{\dot a_i} + \sum_j \frac{1}{\dot b_j}
+ \frac1n + \sum_{i,j} \frac{1}{\dot \cl_{i,j}} \right) < \\
< \frac{1}{6} \left(
\frac{n^{(1-\gamma)} + 1}{n^\gamma - 1} 
+ \frac{n^{(1-\gamma)} + 1}{n^\gamma - 1} + \frac{1}{n}
+ \frac{(n^{(1-\gamma)} + 1)^2}%
{n^{(2\gamma - 1)} - n^{1/20} n^{(\gamma - \frac12)}} \right) \to 0,
\end{gather*}
so the usual Stirling approximation is safe after all.
The next term $T_4$ equals
\[
\sum_i \frac{\log(2\pi)}{2}  
+ \sum_j \frac{\log(2\pi)}{2}
- \left( \frac{\log(2\pi)}{2} \right) 
- \sum_{i,j} \frac{\log(2\pi)}{2} 
= - \frac{(m-1)^2}{2} \log(2\pi). 
\]
Also, since $|\log(n^{-\gamma} \dot a_i)| < 2 n^{-\gamma}$
and $|\log(n^{(1-2\gamma)} \dot \cl_{i,j})| < C n^{(\frac{11}{20} - \gamma)}$,
\begin{gather*}
T_3 = \sum_i \frac{\log \dot a_i}{2}  
+ \sum_j \frac{\log \dot b_j}{2}
- \left( \frac{\log n}{2} \right) 
- \sum_{i,j} \frac{\log \dot \cl_{i,j}}{2} = \\
= \sum_i \frac{\gamma \log n}{2}  
+ \sum_j \frac{\gamma \log n}{2}
- \left( \frac{\log n}{2} \right) 
- \sum_{i,j} \frac{(2\gamma - 1) \log n}{2} + \qquad \\
\qquad + \sum_i \frac{\log(n^{-\gamma} \dot a_i)}{2}  
+ \sum_j \frac{\log(n^{-\gamma} \dot b_j)}{2}
- \sum_{i,j} \frac{\log(1 + \tilde\lh_{i,j} n^{(-\gamma+\frac{1}{2})})}{2} = \\
= \left( \frac{m^2 - 1}{2} - m(m-1)\gamma \right) \log n + o(1).
\end{gather*}
We have
\[ T_2 = - \sum_i \dot a_i - \sum_j \dot b_j +
n + \sum_{i,j} \dot \cl_{i,j} = 0. \]
Simplifications similar to the proof of Lemma \ref{lem:logQ} yields
that $T_1$ equals
\[ \sum_i \xl(\dot a_i) + \sum_j \xl(\dot b_j)
- \xl(n) - \sum_{i,j} \xl(\dot \cl_{i,j}) = 
- \sum_{i,j} \frac{\dot a_i \dot b_j}{n}
\xl\left(1 + \sqrt{\frac{n}{\dot a_i \dot b_j}}\; \tilde\lh_{i,j} \right).
\]

In preparation to use Equation \ref{eq:taylorxl},
$\xl(1+t) = t + t^2/2 + O(t^3)$,
notice that
\[ m = O(n^{(1-\gamma)}), \quad
\left| \sqrt{\frac{n}{\dot a_i \dot b_j}} \; \tilde \lh_{i,j} \right|
= O(n^{\frac{11}{20}-\gamma}). \]
Thus
\[ T_1 = - \sum_{i,j} \sqrt{\frac{\dot a_i \dot b_j}{n}} \; \tilde\lh_{i,j}
- \sum_{i,j} \frac{\tilde\lh_{i,j}^2}{2} + O(n^{(\frac{53}{20} - 3\gamma)})
= - \sum_{i,j} \frac{\tilde\lh_{i,j}^2}{2} + o(1), \]
completing the proof.
\end{proof}

Use the lattice $L_{\bg_\gamma}$ to define the Voronoi decomposition
of the space $V_{\bg_\gamma}$:
\begin{gather*}
V_{\bg_\gamma} = \bigcup_{\tilde l \in L_{\bg_\gamma}} R_{\tilde l}; \\
v \in R_{\tilde l} \quad \iff \quad
\left( \forall\,\tilde l' \in L_{\bg_\gamma}, \;
|v - \tilde l| \le |v - \tilde l'| \right). 
\end{gather*}
Notice that the interiors of the cells $R_{\tilde l}$ are pairwise disjoint.
We only need a crude estimate
of the diameter $\diam(R_{\tilde l})$ of a Voronoi cell.

\begin{lem}
\label{lem:diamvoronoi}
Set $\gamma = 19/20$.
We have \( \diam(R_{\tilde l}) = O(n^{(\frac52 - 3\gamma)}) \).
\end{lem}

\begin{proof}
Given a point $\tilde z \in V_{\bg_\gamma}$,
we obtain a nearby $\tilde l \in L_{\bg_\gamma}$
(usually not the nearest one!).
We shall use the conditions in displays \ref{eq:Vg} and \ref{eq:Lg}
which define $V_{\bg_\gamma}$ and $L_{\bg_\gamma}$.
Set, for $1 \le i, j < m$ (that is,
for entries in the $(m-1) \times (m-1)$ leading principal minor),
$\tilde l_{i,j}$ to be the element of
the set 
\[ \frac{1}{\sqrt{n \hat a_i \hat b_j}}
\left( \ZZ - \frac{\dot a_i \dot b_j}{n} \right) \]
nearest to $\tilde z_{i,j}$.
Use all but two of the equations defining $V_{\bg_\gamma}$,
\[ \sum_j \sqrt{\hat b_j}\, \tilde l_{i,j} = 0, \quad
\sum_i \sqrt{\hat a_i}\, \tilde l_{i,j} = 0, \]
for $1 \le i < m$ and $1 \le j < m$,
to obtain $\tilde l_{i,m}$ and $\tilde l_{m,j}$.
Finally, set $\tilde l_{m,m}$ to satisfy the two remaining equations,
\[ \sum_i \sqrt{\hat a_i}\, \tilde l_{i,m} = 
\sum_j \sqrt{\hat b_j}\, \tilde l_{m,j} = 0, \]
thus defining $\tilde l \in L_{\bg_\gamma}$.
For $1 \le i, j < m$, 
\[ |\tilde z_{i,j} - \tilde l_{i,j}| <
1/\sqrt{n\hat a_i \hat b_j} < 2 n^{(\frac12 - \gamma)}. \]
For $1 \le i < m$, thus,
$|\tilde z_{i,m} - \tilde l_{i,m}| < 2mn^{(\frac12 - \gamma)}
< 4n^{(\frac32 - 2\gamma)}$.
A similar estimate holds for $|\tilde z_{m,j} - \tilde l_{m,j}|$
(with $1 \le j < m$).
Finally, $|\tilde z_{m,m} - \tilde l_{m,m}| < 8n^{(\frac52 - 3\gamma)}$
and, computing the euclidean norm of the difference,
\[ |\tilde z - \tilde l| < 16 n^{(\frac52 - 3\gamma)}, \]
completing the proof.
\end{proof}

Endow the vector space $V_{\bg_\gamma}$ with a unit normal measure:
for an open subset $A \subseteq V_{\bg_\gamma}$, set
\[ \mu(A) = 
(2\pi)^{- \frac{(m-1)^2}{2} }
\int_{A} \exp\left(- \frac{|\tilde w|^2}{2} \right) d\tilde w.  \]

\begin{lem}
\label{lem:goodvoronoi}
Set $\gamma = 19/20$ and $\eta = 1/20$.
For $n \in \NN$, consider the $\gamma$-grid $\bg_\gamma$.
If $\tilde l \in L_{\bg_\gamma,\eta}$ then
\[ P(\bg_\gamma,\tilde l) \sim \mu(R_{\tilde l}). \]
\end{lem}

The $\sim$ symbol above is used in the sense discussed after the statement
of Lemma \ref{lem:prenormal}.
More explicitly: given $\epsilon > 0$ there exists $N$ such that
if $n > N$ and $\tilde l \in L_{\bg_\gamma,\eta}$ then
\[ 1-\epsilon <
\frac{P(\bg_\gamma,\tilde l)}{\mu(R_{\tilde l})} <
1+\epsilon. \]

\begin{proof}
Notice that the condition that $|\tilde\lh_{i,j}| < n^{\eta}$
for all $(i,j)$ implies that
$|\tilde\lh|^2 < m^2 n^{2\eta} < 4 n^{(2 - 2\gamma + 2\eta)}$
and therefore $|\tilde\lh| < 2 n^{(1 - \gamma + \eta)}$.
In this region,
\[ \int_{R_{\tilde\lh}}
\exp\left( - \frac{|\tilde w|^2}{2} \right)
d\tilde w \sim
\Vol(R_{\tilde\lh})
\exp\left( - \frac{|\tilde \lh|^2}{2} \right).
\]
Indeed, for $\tilde w \in R_{\tilde\lh}$,
we have (from Lemma \ref{lem:diamvoronoi})
$|\tilde w - \tilde \lh| < C n^{(\frac52-3\gamma)}$
and $|\tilde \lh|, |\tilde w| < 4 n^{(1 - \gamma + \eta)}$.
Thus $||\tilde w|^2 - |\tilde \lh|^2| <
||\tilde w| + |\tilde \lh||\cdot||\tilde w| - |\tilde \lh|| <
8C n^{(\frac72 - 4\gamma + \eta)} \to 0$,
proving the claim.

Thus, from Lemma \ref{lem:goodnormal},
assuming $\tilde \lh \in L_{\bg_\gamma,\eta}$,
\[ P(\bg_\gamma,\tilde \lh) \sim \tilde C_n\, \mu(R_{\tilde \lh}) =
\tilde C_n\, (2 \pi)^{-\frac{(m-1)^2}{2}}
\int_{R_{\tilde \lh}} \exp\left( - \frac{|\tilde w|^2}{2} \right) d\tilde w \]
for some positive constant $\tilde C_n$.
On the other hand, we know from Lemma \ref{lem:begood} that
\[ \sum_{\tilde \lh \in L_{\bg_\gamma,\eta}} \mu(R_{\tilde \lh})
\to 1,  \quad
\sum_{\tilde \lh \in L_{\bg_\gamma,\eta}} P(\bg_\gamma,\tilde \lh) \to 1.  \]
Thus, $\tilde C_n$ can be taken to be $1$, completing the proof.
\end{proof}

Notice that we computed the approximate value of $\Vol(R_{\tilde l})$:
\[ \Vol(R_{\tilde l}) \sim
n^{\left( \frac{m^2 - 1}{2} - m(m-1)\gamma \right)}. \]

We are ready to present the proof of Theorem \ref{theo:discrete}.

\begin{proof}
We describe the space $\Omega_n$ and define the random variables
$X_n$ (which obtains a discrete copula) and $f$ (the bridged Brownian sheet).
First, we obtain random variables
$\tilde z \in V_{\bg_\gamma}$ and $\tilde h \in L_{\bg_\gamma}$
(towards samplings of $f$ and $\sqrt{n}\; \by_n$) with
\[ \Pr_{\Omega_n}[\tilde z \in A] = \mu(A),\quad
A \subset V_{\bg_\gamma}; \qquad
\Pr_{\Omega_n}[\tilde h = \tilde l] = P(\bg_\gamma,\tilde l). \]
These two random variables are far from independent.
In a way, we would like to always have $\tilde z \in R_{\tilde h}$
but that is not quite possible since in principle
\( \mu(R_{\tilde l}) \ne  P(\bg_\gamma,\tilde l) \).

\begin{figure}[t]
\begin{center}
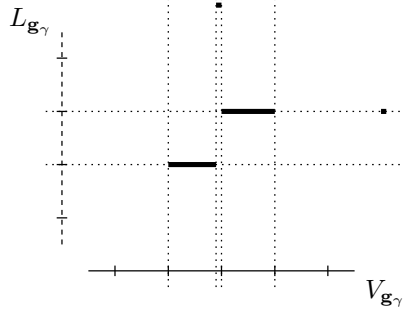
\end{center}
\caption{Projecting $\Omega_n$
on $V_{\bg_\gamma}$ and $L_{\bg_\gamma}$}
\label{fig:Omegan}
\end{figure}

Figure \ref{fig:Omegan} shows the beginning of the construction
of $\Omega_n$ and these two random variables.
Sometimes, as in the figure, we have 
\( \mu(R_{\tilde l}) >  P(\bg_\gamma,\tilde l) \).
For some values of $\tilde z \in R_{\tilde l}$,
we must therefore have $\tilde h \ne \tilde l$.
In the second step,
\( \mu(R_{\tilde l}) <  P(\bg_\gamma,\tilde l) \)
and, for some values of $\tilde z$ in other Voronoi cells,
we have $\tilde h = \tilde l$.
We know, however, from Lemmas \ref{lem:begood} and \ref{lem:goodvoronoi}
that, following this construction,
\[ \lim_{n \to \infty} \Pr_{\Omega_n}[\tilde z \in R_{\tilde h}] = 1. \]

Let $u_i = v_i = a_i/n$ so that $\bg_\gamma$ can be interpreted
as a grid in the unit square $[0,1]^2$. 
Given $\tilde h = \tilde l \in L_{\bg_\gamma}$,
uniformly select $X_n$ from the set of discrete copulas
$C$ for which, as in Equation \ref{eq:tildeh},
\[ \hat \cl_{i,j} = \hat a_i \hat b_j +
\sqrt{\frac{\hat a_i \hat b_j}{n}}\; \tilde l_{i,j} \]
and therefore
\[ \sqrt{n}\;\by_n(u_i,v_j) = 
\sum_{i'\le i; \; j'\le j}
\sqrt{\hat a_{i'} \hat b_{j'} }\;\tilde h_{i',j'}. \]
As in Equation \ref{eq:jointbrown},
set the values of $f$ on $\bg_\gamma$ to be
\[ f\left(u_i,v_j\right) =
\sum_{i'\le i; \; j'\le j}
\sqrt{\hat a_{i'} \hat b_{j'} }\;\tilde z_{i',j'}. \]
From this point on, obtain a bridged Brownian sheet $f$ in the usual way.
This completes the construction of $\Omega_n$ and of the random variables
$X_n$ and $f$. Items (a) and (b) in the statement of the theorem
are trivial; we are left with proving (c).

We may disregard subsets of $\Omega_n$
whose probability tends to zero (as a function of $n$).
We may therefore assume that $\tilde z \in R_{\tilde h}$ and,
from Lemma \ref{lem:diamvoronoi},
that $|\tilde z - \tilde h| < n^{(\frac52 - 3\gamma)}$.
For a point $(u_i,v_j)$ on the grid $\bg_{\gamma}$ we have
\[ \sqrt{n}\;\by_n(u_i,v_j) - f\left(u_i,v_j\right) =  
\sum_{i'\le i; \; j'\le j}
\sqrt{\hat a_{i'} \hat b_{j'} }\;(\tilde h_{i',j'} - \tilde z_{i',j'}) \]
and therefore, from Lemma \ref{lem:diamvoronoi},
\begin{equation} 
\label{eq:th11}
\left| \sqrt{n}\;\by_n(u_i,v_j) - f\left(u_i,v_j\right) \right|
= \left| \langle M, (\tilde h - \tilde z) \rangle \right| 
\le 2 |(\tilde h - \tilde z)| 
\le 2 n^{(\frac52-3\gamma)}.
\end{equation} 
Here $M_{i',j'} = \sqrt{\hat a_{i'} \hat b_{j'}}$ for $i'\le i$, $j' \le j$
and $0$ otherwise; clearly, $|M|^2 = \sum (M_{i',j'})^2 \le 4$.

For an arbitrary point $(u,v) \in \II^2$,
let $(u_i,v_j)$ be the point in the grid $\bg_\gamma$ closest to $(u,v)$.
Apply Lemma \ref{lem:holder} with $\alpha = \gamma = 19/20$
and $\epsilon = 1/100$ to deduce that, a.a.s.,
$X_n$ is $(\alpha,\epsilon)$-H\"older and therefore,
for all $(u,v)  \in \II^2$,
\begin{equation} 
\label{eq:th12}
| \sqrt{n}\;\by_n(u,v) - \sqrt{n}\;\by_n(u_i,v_j) |
< n^{(\frac{\gamma}{2} + \epsilon - \frac12)}. 
\end{equation} 
Furthermore, from Lemma \ref{lem:holderbrownsheet}
with $r = 4 n^{-\frac{1}{20}}$, $\delta = \frac15$ and $C = \frac14$,
a.a.s. we have that, for all $(u,v)  \in \II^2$,
\begin{equation} 
\label{eq:th13}
| f(u,v) - f(u_i,v_j) | < n^{(\frac{\gamma}{2} + \epsilon - \frac12)}. 
\end{equation} 
Thus, from Equations \ref{eq:th11}, \ref{eq:th12} and \ref{eq:th13},
we a.a.s. have
\[ | \sqrt{n} \by_n - f |_{C^0} <
2 n^{(\frac{\gamma}{2} + \epsilon - \frac12)} + 2 n^{(\frac52-3\gamma)}
\to 0, \]
completing the proof.
\end{proof}

\section{A concentration result}
\label{sect:concentration}

In this section we present some applications of the main theorem.
We recall some basic statistics
for a bridged Brownian sheet $f \in C^0_0(\II^2)$
(\cite{DudleyRAP}, \cite{DudleyUCLT}).
For any fixed $p \in [1, +\infty]$
\[ \Pr\left[ \| f \|_{L^p} < r \right] = \rho^p(r) \]
defines a strictly increasing continuous function
$\rho^p: [0,+\infty) \to \II$
with
\[ \lim_{r \to 0} \rho^p(r) = 0; \quad
\lim_{r \to +\infty} \rho^p(r) = 1. \]
Theorem \ref{theo:discrete} obtains similar results for discrete copulas.
For $C \in \cS_n$ and $p \in [1,+\infty)$ define
\[ d^p(C) =
\left( \frac{1}{n^2} \sum_{a,b} \left| \frac{C_{a,b} - \frac{ab}{n}}{\sqrt{n}}
\right|^p \right)^{1/p}; \qquad
d^\infty(C) = \max_{a,b} \left| \frac{C_{a,b} - \frac{ab}{n}}{\sqrt{n}}
\right|. \]

\begin{cor}
\label{cor:rho}
Let $p \in [1,+\infty]$ and $r \in (0,1)$ be fixed.
Then
\[ \lim_{n\to \infty} \Pr\left[
d^p(X) < r
\right] = \rho^p(r). \]
\end{cor}

\begin{proof}
Theorem \ref{theo:discrete} directly yields the case $p = \infty$.
For the general case, notice that
$\|\sqrt{n}\;\by_n - f\|_{C_0} < \epsilon$ implies
$\|\sqrt{n}\;\by_n - f\|_{L^p} < \epsilon$.
\end{proof}

A weaker statement is the following concentration result,
which makes no reference to bridged Brownian sheets.

\begin{cor}
\label{cor:nobrown}
Let $g: \NN \to \RR$ be a function with $\lim g(n) = +\infty$.
Let $X_n$ be a random discrete copula in $\cS_n$.
Then
\[ \lim_{n \to \infty}
\Pr\left[ \left\| \by_n \right\|_{C^0} >
\frac{g(n)}{\sqrt{n}} \right] = 0. \] 
\end{cor}

\begin{proof}
The result follows from Corollary \ref{cor:rho} and
$\lim_{r\to+\infty} \rho^{\infty}(r) = 1$.
\end{proof}

\section{Discrete copulas and lozenge tilings}
\label{sect:tilings}

The graph of a discrete copula can also be seen as a pile of cubes.
Figure \ref{fig:arctic} exemplifies this correspondence
for a permutation in $S_8$.
The north-east corner of the square on the left is taken
to the south vertex of the bottom of the cube.
The bullets correspond to the black squares.

\begin{figure}[ht]
\begin{center}
\includegraphics[height=50mm,angle=0]{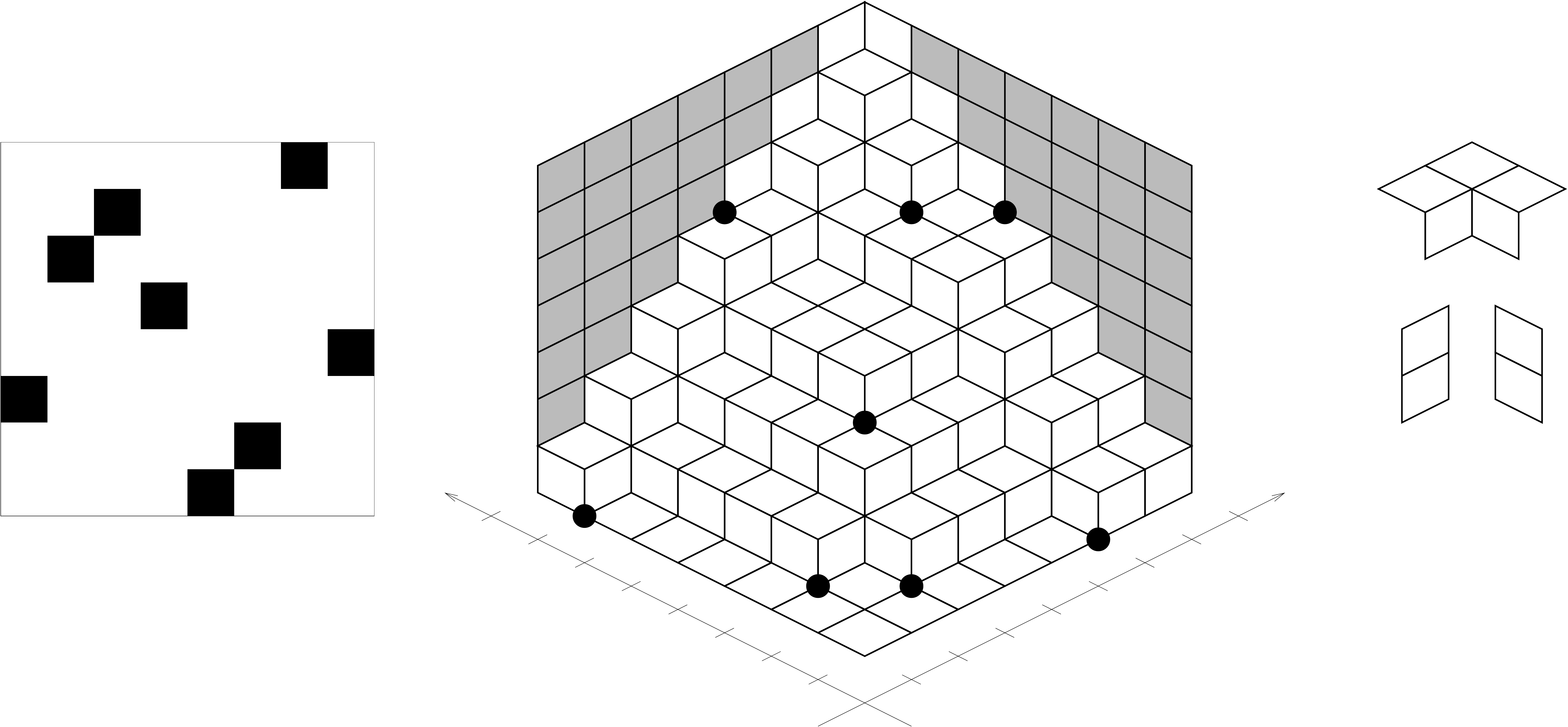}
\end{center}
\caption{A permutation and its representation as a lozenge tiling
of a subset of the hexagon; the configurations to the right are forbidden.}
\label{fig:arctic}
\end{figure}

This defines a bijection between $\cS_n$
and a certain set of tilings of a regular hexagon of side $n-1$.
Clearly, not every gravitationally stable pile of cubes corresponds
to a permutation:
the grey walls are forced and the three configurations
to the right of Figure \ref{fig:arctic} are forbidden
(in the white part).
It is not hard to prove that these conditions characterize
tilings corresponding to copulas.

Theorem \ref{theo:discrete} implies
that if we randomly select a tiling corresponding to a copula
the surface of the pile of cubes is essentially
an affine transformation of a Brownian sheet.
The reader may contrast this result with
the celebrated Arctic circle theorem  (see \cite{JPS}):
in our scenario there are no ``frozen'' regions.

\section{Birkhoff copulas}
\label{sect:birkhoff}

A \textit{Birkhoff copula} is a matrix $C$ with \textit{real}
entries satisfying conditions (i), (ii) and (iii)
in the definition of discrete copula.
Discrete copulas are associated with permutation matrices.
The more general Birkhoff copulas
correspond to doubly stochastic matrices.
More precisely, the \textit{Birkhoff polytope} $B_n$ is
the convex set of $n \times n$ doubly stochastic matrices,
whose vertices are the permutation matrices.
The counterpart to $\cS_n$ is its convex hull $\cB_n$,
the polytope of Birkhoff copulas.
The natural bijection $B_n \to \cB_n$ takes $M$  to $C$,
where again $C$ is obtained by integration:
\[ C_{a,b} = \sum_{a' \le a; \; b' \le b} M_{a',b'}\,. \]
Thus, a Birkhoff copula indicates the joint distribution
of two discrete random variables $Z_0$ and $Z_1$,
both uniform in $\{1, 2, \ldots, n\}$ but usually not independent.

\medskip

For the Birkhoff case,
numerical experiments (using a Metropolis style algorithm)
obtain uniformly distributed random
elements of $\cB_n$, the set of Birkhoff copulas.
We then study their statistics, particularly the analogue
of Theorem \ref{theo:joint}.
The results ratify the following conjecture,
a counterpart of Theorem \ref{theo:discrete}.


\begin{conj}
\label{conj:birkhoff}

There exist probability spaces $\Omega_n$ with the following properties.
\begin{enumerate}[(a)]
\item{On $\Omega_n$ there exists a random variable
$X_n$ uniformly distributed in $\cB_n$.}
\item{On $\Omega_n$ a sample-continuous bridged Brownian sheet process
$f$ in $C^0_0(\II^2)$ is defined.}
\item{For all $\epsilon > 0$,
\( \lim_{n \to \infty}
\Pr\left[ \left\| n\;\by_n - f \right\|_{C^0} >
\epsilon \right] = 0. \) }
\end{enumerate}
\end{conj}

More precisely: we select points $(u_1,v_1), \ldots, (u_m,v_m) \in (0,1)^2$
and compare the empirically observed joint distribution
of $\by_n(u_i,v_i)$ with the joint normal distribution
predicted by the conjecture: they match.

\goodbreak


\bigskip\bigskip\bigbreak

{

\parindent=0pt
\parskip=0pt
\obeylines

Juliana Freire, Nicolau C. Saldanha and Carlos Tomei,  PUC-Rio
jufreire@gmail.com
saldanha@puc-rio.br; http://www.nicolausaldanha.com
tomei@mat.puc-rio.br

\smallskip

Departamento de Matem\'atica, PUC-Rio
R. Marqu\^es de S. Vicente 225, Rio de Janeiro, RJ 22453-900, Brazil

}

\end{document}